\documentclass[11pt]{article}
\input epsf.tex
\usepackage[T1]{fontenc}
\usepackage{ae}
\usepackage{amssymb,amsmath,amsthm,textcomp,amsxtra}
\usepackage[mathcal]{euscript}
\usepackage{graphicx,pst-plot,mparhack}
\usepackage[abs]{overpic}
\usepackage[matrix,arrow]{xy}
  \xymatrixrowsep{0.3pc}
  \xymatrixcolsep{0.3pc}
\usepackage{enumerate,makeidx}

\renewcommand*{\geq}{\geqslant}
\renewcommand*{\leq}{\leqslant}

\makeindex

\numberwithin{equation}{section}

\newcommand*{\abs}[1]{\lvert#1\rvert}

\newcommand*{\Sscr}{\mathcal S}

\newcommand*{\N}{\mathbb{N}}
\newcommand*{\R}{\mathbb{R}}

\renewcommand*{\d}{\mathrm{d}}

\DeclareMathOperator{\var}{var}

\usepackage{amsfonts,amsmath,latexsym,amsthm,amssymb}
\usepackage{dsfont} 
\usepackage{mathtools}
\usepackage{enumerate}
\usepackage{mathrsfs} 
\mathtoolsset{showonlyrefs}
\usepackage{listings} 
\usepackage{amscd}
\DeclareMathOperator{\const}{const}

\def\wt{\widetilde}
\def\pas{\ \d P \mbox{-a.s.}}
\def\<{\left<}
\def\>{\right>}
\def\({\left(}
\def\){\right)}
\def\vp{\varphi}
\def\o{\omega}
\def\O{\Omega}
\def\9{\infty}
\def\R{\mathbb R}
\def\N{\mathbb N}

\def\shb{{\cal B}}
\def\shc{{\cal C}}

\def\she{{\cal E}}
\def\shf{{\cal F}}
\def\shg{{\cal G}}
\def\shh{{\cal H}}
\def\shm{{\cal M}}

\def\shs{{\cal S}}
\def\shy{{\cal Y}}
\def\shw{{\cal W}}

\newcommand{\norm}[1]{\left\| #1 \right\|}

\newtheorem{theo}{Theorem}[section]
\newtheorem{lemma}[theo]{Lemma}
\newtheorem{Assumption}[theo]{Assumption}

\newtheorem{prop}[theo]{Proposition}
\newtheorem{rem}[theo]{Remark}

\newtheorem{defi}[theo]{Definition}

\newcommand{\beqnar}{\begin{eqnarray*}}
\newcommand{\eeqnar}{\end{eqnarray*}}
\newcommand{\ba}{\begin{array}}
\newcommand{\ea}{\end{array}}

\newcommand{\halb}{\frac{1}{2}}

\begin{document}

\title{A stochastic Fokker-Planck equation and 
double probabilistic representation for the
stochastic  porous media type equation.}

\author{ Viorel Barbu (1), Michael R\"ockner (2)
and Francesco Russo (3) } 

\date{April 18th  2014}
\maketitle

\thispagestyle{myheadings}
\markright{Stochastic porous media with multiplicative noise.} 

{\bf Summary:} The purpose of the present paper consists in proposing 
and discussing a double probabilistic representation
for a porous media equation in the whole space
perturbed by a multiplicative colored noise.
For almost all random realizations $\omega$, one associates
a stochastic differential equation in law with random coefficients,
driven by an independent Brownian motion.
The key ingredient is a uniqueness lemma for a 
linear SPDE of Fokker-Planck type with measurable bounded
(possibly  degenerated) random coefficients.


{\bf2000  AMS-classification}: 35R60, 60H15, 60H30, 60H10, 60G46,
 35C99, 58J65, 82C31.



\begin{itemize}
\item[(1)] Viorel Barbu,
University Al.I. Cuza, Ro--6600 Iasi,
Romania.
\item[(2)] Michael R\"ockner,
Fakult\"at f\"ur Mathematik, 
Universit\"at   Bielefeld, 
\\ D--33615 Bielefeld, Germany
\item[(3)] Francesco Russo,
ENSTA ParisTech,
Unit\'e de Math\'ematiques appliqu\'ees,
 828, boulevard des Mar\'echaux,
F-91120 Palaiseau (France).

\end{itemize}

\vfill \eject

\section{Introduction}

\setcounter{equation}{0}

We consider a function $\psi: \R \rightarrow \R$ and
real functions $e^1, \ldots, e^N$ on $\R$, for some strictly positive integer 
$N > 0$. In the whole paper, the following Assumption will be in force.
\begin{Assumption}\label{E1.0}
\begin{itemize}
\item  $\psi : \R \rightarrow \R$ is such that
its restriction to ${\R_+}$
is monotone increasing, with   $\psi(0) = 0$.
\item 
$\vert \psi (u) \vert \le  {\rm const} \vert u \vert, \ u \ge 0.  $ \\
In particular,  $\psi$ is right-continuous at zero and $\psi(0) = 0$.
\item Let  $e^i \in  C^2_{\rm b}(\R), 0 \le i \le N,$ such that  they are $H^{-1}$-multipliers
in the sense that the maps $\varphi \mapsto \varphi e^i$
are continuous in the $H^{-1}$-topology.   $\shc(e^i)$
denotes the norm of this operator and we will call it {\bf multiplier norm}.
\end{itemize}
\end{Assumption}
 Let $ T > 0$ and $(\Omega, \shf, P)$,
be a fixed  probability space.
Let $(\shf_t, t \in [0,T])$ be a filtration
fulfilling the usual conditions and 
we suppose $\shf = \shf_T$.
Let $\mu(t,\xi), t \in [0,T], \xi \in \R,$ be a random field
of the type
$$ \mu(t,\xi) = \sum_{i=1}^N e^i (\xi) W^i_t + e^0(\xi) t,
 \ t \in [0,T], 
\xi \in \R,$$
where $W^i, 1 \le i \le N,$ are independent continuous
 $\shf_t)$-Brownian motions   on $(\Omega, \shf, P)$, 
 which are fixed from now on until the end of the paper.
For technical reasons we will sometimes set
$W^0_t \equiv t$.
We focus on a stochastic partial differential equation 
of the following type:
\begin{equation}
\label{PME}
\left \{
\begin{array}{ccl}
\partial_t X(t,\xi)&=& \frac{1}{2} \partial_{xx}^2(\psi(X(t, \xi) ) +
 X(t,\xi) \partial_t \mu(t, \xi),
\\
X(0,\d \xi)& = & x_0(\d \xi), 
\end{array}
\right.
\end{equation}
which holds in the sense of Definition \ref{DSPDE}, where $x_0$ is a
a given probability measure on $\R$.
The stochastic multiplication above is of It\^o type.
We look for a solution of (\ref{PME}) with time evolution in $L^1(\R)$.

\begin{rem}\label{Rint} 
\begin{enumerate}
\item With $\psi$ we can naturally associate an odd increasing function
$\psi_0: \R \rightarrow \R$ such that $\psi_0(u) = - \psi_0(-u)$
for every $u \in \R$.  
\item By the usual technique of {\it filling the gap,}  $\psi$ 
can be associated with a graph, i.e. a multivalued function
$ \R \mapsto 2^\R$,
still denoted by the same letter, by
 setting $\psi(u)  = [\psi(u-), \psi(u+)]$.
\end{enumerate}
 \end{rem}
Since $\psi$ restricted to $\R_+ $ is monotone, Assumption \ref{E1.0} implies
   $ \psi (u) = \Phi^2(u) u, \ u  > 0$, $\Phi: \R^\star_+ \rightarrow \R$
   being a
non-negative Borel function which is  bounded on $\R^\star_+$. 
When $\psi (u) = \vert u \vert^{m-1}u$,
 $m  >   1$, (\ref{PME}) and $\mu \equiv 0$, 
\eqref{PME}
 is  nothing else but the
classical {\it porous media equation}.
When $\psi $ is a general increasing function
(and $\mu \equiv 0$),
there are several contributions to 
the analytical study of (\ref{PME}),
 starting from
\cite{BeBrC75} for existence,   \cite{BrC79} for  uniqueness in the
case of bounded solutions 
 and \cite{BeC81} 
for continuous dependence on the coefficients.
The authors consider  the case where $\psi$ is  continuous, even
if their arguments allow some extensions for the discontinuous case.
Those are the classical references when the space variable varies
on the real line. For equations in a bounded domain and Dirichlet
boundary conditions, for simplicity, we only refer to monographs,
e.g. \cite{vazquez, Show, barbu93, barbu10}.

As far as the stochastic porous media is concerned,
most of the work for existence and uniqueness
concerned the case of bounded domain, see for instance
\cite{BDR09, BDR09CMP, BDR08}.
The infinite volume case, i.e. when the underlying domain is $\R^d$, 
was fully analyzed in  \cite{Ren}, when
 $\psi$ is polynomially bounded
 (including the fast diffusion case) 
when the space dimension is $d \ge 3$,
more precisely, see \cite{Ren}, Theorem 3.9, Proposition 3.1 and 
Exemple 3.4.
To the best of  our knowledge, except for  \cite{Ren} 
and our companion paper  \cite{BRR3}, 
 this seems to be  the only work concerning a stochastic porous type equation
in infinite volume.




\begin{defi} \label{DNond} 
\begin{itemize} 
\item We will say that equation (\ref{PME}) (or $ \psi$) is 
{\bf non-degenerate} if on each compact,
 there is a constant $c_0 > 0$ such that
$ \Phi  \ge c_0  $.
\item We will say that equation (\ref{PME}) or $ \psi$ is 
{\bf degenerate} if  $ \lim_{u \rightarrow 0_+} \Phi(u) = 0  $
  in the sense
that for any sequence of non-negative reals $(x_n)$ converging
to zero, and $y_n \in \Phi(x_n)$ we have 
$\lim_{n \rightarrow \infty} y_n = 0$, see Remark \ref{Rint} 2.
\end{itemize} 
\end{defi}
 One of the typical examples of degenerate $\psi$ is the
case of $\psi$ being {\bf strictly increasing after some zero}.
This notion was introduced in \cite{BRR2} and it means the following.
There is $0 \le u_c $ such that $\psi_{[0,u_c]} \equiv 0$ and
$\psi$ is strictly increasing on $]u_c, +\infty[$.

\begin{rem}\label{DegNonD}
\begin{enumerate}
\item 
 $\psi$ is non-degenerate if and only if
$\liminf_{u \rightarrow 0+} \Phi(u)  ( = \lim_{u \rightarrow 0+} \Phi(u)) > 0 $.
\item Of course, if $\psi$ is strictly increasing  after 
some zero, with $u_c > 0$ then
 $\psi$ is degenerate.
 If  $\psi$ is degenerate, then
$\psi^\kappa (u) =  (\Phi^2(u) + \kappa) u $,
for every $\kappa > 0$, is non-degenerate.
In the sequel we  will set $\Phi(0) := \lim_{u \downarrow 0} \Phi(u)$.
In particular, if $\psi$ is degenerate we 
have $\Phi(0) = 0$.
 \end{enumerate}
\end{rem}
This paper will be devoted to both the case when $\psi$ 
is non-degenerate and the case when $\psi$ is 
degenerate.\\



One of the targets of the present paper
concerns the   probabilistic representation of solutions to \eqref{PME}
extending the results of \cite{BRR1, BRR2} which treated the 
deterministic case $\mu \equiv 0$.
In the deterministic case,
to the best of our knowledge the first author who considered a
probabilistic representation (of the type studied in this paper) for the
solutions of a non-linear deterministic PDE was McKean
\cite{mckean}, particularly in relation with the so called propagation of
chaos. In his case, however, the coefficients were smooth. From then on
the literature steadily grew and nowadays there is a vast amount of
contributions to the subject, see the reference 
list of \cite{BRR1, BRR2}.
A probabilistic representation
 when  $\beta(u) = \vert u \vert  u^{m-1}, m  > 1,$ was provided for
instance in \cite{BCRV}, in
 the case of the classical 
porous media equation. When $m < 1$, i.e. in the case of the fast
diffusion equation, \cite{BR} provides a probabilistic representation of the
so called {\bf Barenblatt solution}, i.e. the solution whose initial condition 
is concentrated at zero.
 
\cite{BRR1, BRR2} discussed the probabilistic representation
when $\mu = 0$  in the non-degenerate and degenerate case respectively, 
where $\psi$ also may have jumps. In the sequel of this introduction
we will suppose $\psi$ to be single-valued. 

In the case $\mu =  0$, the equation \eqref{PME}
models a non-linear phenomenon macroscopically.
Let us denote by $u:[0,T] \times \R \rightarrow \R$ the
solution of that equation.
The idea of the probabilistic representation is
to find a process $(Y_t, t \in [0,T])$ whose law 
at time $t$ has for  density $u(t,\cdot)$.

 $Y$ turns out to be the weak solution of the 
non-linear stochastic differential equation
\begin{equation}
\label{E1.2}
\left \{
\begin{array}{ccc}
Y_t &=& Y_0 + \int_0^t \Phi(u(s,Y_s)) dB_s,  \\
{\rm Law } (Y_t) &=& u(t,\cdot), \quad t \ge 0, \\
\end{array}
\right.
\end{equation}
where $B$ is a classical Brownian motion.
The behaviour of $Y$ is the microscopic counterpart of
the phenomenon described by (\ref{PME}), describing
the evolution of  a single particle, whose law 
behaves according to  (\ref{PME}).




The idea of this paper is to consider the case when $\mu \neq 0$.
 This includes the case when 
the $\mu$ is not vanishing but it is deterministic; it happens
when only $e^0$ is non-zero, and $e^i \equiv 0, 1 \le i \le n$.
 In this case our technique gives a sort of forward
Feynman-Kac formula for non-linear PDEs.

We introduce a double
stochastic representation (in a strong-weak probabilistic sense)
by means of introducing an enlarged probability space 
on which one can represent the solution of \eqref{PME} 
as the (generalized)-law (called $\mu$-law) of
a solution to a non-linear SDE.
Intuitively, it describes the microscopic aspect of the SPDE \eqref{PME}
for almost all quenched $\omega$. 
The terminology strong refers to the case that the 
probability space $(\Omega, \shf, P)$ on which the SPDE
is defined, will remain fixed.

We represent a solution $X$ to  \eqref{PME} making use of
another independent source of randomness described by another
probability space based on some set $\Omega_1$. 

The analog of the process $Y$, obtained 
when $\mu$ is zero in \cite{BRR2, BRR1},
is a doubly stochastic process, still denoted
by $Y$ defined on $(\Omega_1 \times \Omega, Q)$,
for which, $X$ constitutes the so-called
family of {\it $\mu$-marginal laws of $Y$.}
More precisely, for fixed $\omega \in \Omega$, the $\mu$-marginal 
law at time $t$
of  process $Y$ is given
by the positive finite Borel measure 
\begin{equation} \label{1.2ter}
 A \mapsto E^{Q^\omega}\left(1_{A}(Y_t) \she_t \left
(\int_0^\cdot \mu(ds, Y_s(\cdot,\omega))\right)
\right), 
\end{equation}
$\she$ denoting the Dol\'eans exponential,
where 
\begin{equation} \label{1.2bis}
\int_0^t \mu(ds, Y_s(\cdot, \omega)) = \sum_{i=0}^N
 \int_0^t e^i(Y_s(\cdot, \omega)) dW^i_s(\omega), t  \in [0,T],
\end{equation}
and where we assume that for some filtration $(\shg_t)$ on
$\Omega_1 \times \Omega$, $Y$ is $(\shg_t)$-adapted 
and $W^1, \ldots, W^N$ are $(\shg_t)$-martingales on 
$\Omega_1 \times \Omega$.
For fixed $ t \in [0,T]$ we also say that the
previous measure is the {\it $\mu$-law} of $Y_t$.
In the case $e^0 = 0$, the situation is the following.
 For each fixed $\omega \in \Omega$,
\eqref{1.2ter}  is a (random) non-negative measure
which is not a (random) probability. 
 However the expectation of its 
 total mass is indeed $1$.

The  double probabilistic representation is based on a simple idea.
Suppose there is a process $Y$ defined on a suitably enlarged
probability space $(\Omega_1 \times \Omega, Q)$ such that 
\begin{equation}
\label{DPIY}
\left \{
\begin{array}{ccc}
Y_t &=& Y_0 + \int_0^t \Phi(X(s,Y_s)) dB_s, \\
\mu-{\rm Law } (Y_t) &=& X(t,\xi) \d  \xi, \quad t \in ]0,T],\\
\mu-{\rm Law } (Y_0) &=& x_0(\d \xi), 
\end{array}
\right.
\end{equation}
where $B$ is a standard Brownian motion.
Then  $X$  solves the SPDE \eqref{PME}. This is the object of Theorem \ref{T32}.
Vice versa, if $X$ is a solution of \eqref{PME} then there is a process $Y$
solving \eqref{DPIY}, see Theorem \ref{T73}.

\begin{rem} \label{RDPIY}
\begin{enumerate}
\item If $X$ is a solution of \eqref{PME}, then \eqref{DPIY} implies that    
$ X \ge 0 \ dt \otimes dx \otimes dP$ a.e. 
\item Let $t \in [0,T]$. Let $\varphi: \R \rightarrow \R$ be
Borel and bounded. Then
$$ \int_\R \varphi(\xi)  X(\omega)(t,\xi) \d \xi = 
 E^{Q^\omega}\left(\varphi(Y_t(\omega)) \she_t\left(\int_0^\cdot \mu(ds, Y_s(\omega))\right)
\right).$$
So 
$$ \int_\R   X(\omega)(t,\xi) d\xi = 
 E^{Q^\omega} \left(\she_t \left(\int_0^\cdot \mu(ds, Y_s(\omega))\right)
\right).$$
Even though  for a.e. $\omega \in \Omega$,  the   previous expression
is not necessarily a probability measure, of course, 
$$ \nu_\omega: \varphi \mapsto \frac{\int_\R \varphi(\xi) 
 X(\o)(t,\xi) d\xi}
{\int_\R  X(\omega)(t,\xi) d\xi } $$
is one. It can be expressed as 
$$\nu_\omega(A) =  \frac{E^{Q^\omega}(1_A(Y_t) \she_t(M(\cdot, \omega)))}
{E^{Q^\omega} \she_t(M(\cdot,\omega))}, $$
where
$M_t(\cdot,\omega) = \int_0^t \mu(ds, Y_s(\cdot, \omega)), t \in [0,T],$
 is defined in \eqref{1.2bis}.
\item Consider the particular case $e_0 = 0, e_1 = c$, 
$c$ being some constant. In this case, the $\mu$-marginal laws are given by
\begin{eqnarray*}
A &\mapsto & E^{Q^\omega}(1_A(Y_t) c \she_t(W)) = 
 c \she_t(W) E^{Q^\omega}(1_A(Y_t)) \\ &=
& c \she_t(W) \nu_\omega(t,A)
\end{eqnarray*}
and $\nu_\omega(t,\cdot)$ is the law of $Y_t(\cdot,\omega)$ under 
$Q^\omega$.
\end{enumerate}
\end{rem}
\begin{rem} \label{RFiltering}
Item 2. of Remark \ref{RDPIY} has a filtering interpretation, see
e.g. \cite{pardoux} for a comprehensive introduction. \\
Suppose $e^0 = 0$.
Let $\hat Q$ be a probability on $(\Omega, \shg_T, \hat Q)$,
and consider the non-linear diffusion problem \eqref{E1.2}
as a basic dynamical phenomenon. We suppose now  
that there are N observations $Y^1,\ldots,Y^N$ related to  the process $Y$
generating a filtration $(\shf_t)$.
We suppose in particular 
that $\d Y^i_t = \d W^i_t + e^i(Y_t), 1 \le i \le N,$
and  $W^1, \ldots, W^N$ be $(\shf_t)$-Brownian motions.
Consider the following dynamical system of non-linear diffusion type:
\begin{equation}\label{EFiltering1}
\left \{
 \begin{array}{ccc}
Y_t &=& Y_0 + \int_0^t \Phi(X(s,Y_s))dB_s \\
\d Y^i_t &=& \d W^i_t + e^i(Y_s), 1 \le i \le N,\\
X(t,\cdot)&:& {\rm conditional \ law \ under} \ \shf_t.  
\end{array}
\right.
\end{equation}
The third equality of \eqref{EFiltering1} means, under $\hat Q$, that we have
\begin{equation} \label{EFiltering2}
\int_\R \varphi(\xi) X(t,\xi) \d \xi = 
E(\varphi(Y_t) \vert \shf_t).
\end{equation}
We remark that, under the new probability $Q$
defined by $\d Q = \d \hat Q \she(\int_0^T \mu(\d s, Y_s))$,
 $Y^1, \ldots, Y^N$ are standard $(\shf_t)$-independent Brownian motions.
Then \eqref{EFiltering2} becomes
\begin{eqnarray*}
\int_\R \varphi(\xi) X(t,\xi) \d \xi &=& E^{\hat Q}(\varphi(Y_t) \vert \shf_t) \\
&=&  \frac{E^{Q}(\varphi(Y_t)   \she_t (\int_0^\cdot \mu(ds, Y_s) \vert \shf_t ))}
{E^{Q}(\she_t (\int_0^\cdot \mu(ds, Y_s) \vert \shf_t))}.
 \end{eqnarray*}
Consequently by Theorem \ref{T32} $X$ will be the solution 
of the SPDE \eqref{PME}, with $x_0$ being the law of $Y_0$; so  
\eqref{PME} constitutes
the Zakai type equation associated with our filtering problem.
\end{rem}

The present approach has some vague links with the topic of random irregular 
media. In this case the macroscopic equation is a linear random partial 
differential equation with diffusion being $1$ and with a drift
which is the realisation of a Brownian motion $W$. Here the equation
has a random second term, and the diffusion term is non-linear.
Our type of stochastic differential equation  \eqref{E1.2} is
time inhomogeneous (and depending on the law
of the solution),  in contrast to the one of random media.
 The literature of random (even irregular) media is huge,
see for instance \cite{mathieu, hushi}.
We mention that a stochastic calculus approach in  this context
was developed in \cite{FRW1,FRW2,RTrutnau}.

The paper is organized as follows. After the present introduction,
Section \ref{S2} is devoted to preliminaries, to the notion
of a $\mu$-law associated  with a (doubly) stochastic process and 
to the notion of
weak-strong solution of a doubly random stochastic differential equation.
In 
Section \ref{S3} we define the notion of double probabilistic
representation and we
expose the main idea behind it.
 Section \ref{S4} shows that
the $\mu$-law of the solution of a {\it non-degenerate} weak-strong
stochastic differential equation, always admits a density for a.e. $\omega \in 
\Omega$. In Section \ref{S5} 
 a uniqueness theorem for an SPDE of Fokker-Planck type is
 formulated and proved, which is useful for the 
weak-strong probabilistic representation when $\psi$
is non-degenerate, but it has an interest
in itself.
Section \ref{S6}  shows existence and uniqueness
 of the double stochastic
representation of \eqref{PME} when $\psi$ is non-degenerate and
 finally Section \ref{S7}
provides the double probabilistic representation 
when $\psi$  is degenerate and 
 strictly increasing after some zero,
 in case  $\psi$ is Lipschitz.
We believe that the latter assumption can be generalized, which is the subject of future work.

\section{Preliminaries} 

\setcounter{equation}{0}

\label{S2}

First we introduce some basic recurrent notations.
$C_0^\infty(\R)$ is the space of smooth functions with compact support.
$H^{-1}(\R)$ is the classical Sobolev space.
$\shm(\R)$ (resp. $\shm_+(\R)$) denotes the space of
finite real (resp. non-negative) measures. 
 \\
We recall that  $\shs(\R)$ is the space of the Schwartz fast decreasing
test functions.  $\shs'(\R)$ is its dual, i.e. 
the space of Schwartz tempered distributions.
On $\shs'(\R)$, the map $(I-\Delta)^\frac{s}{2},  s \in \R,$ is well-defined.
For $s \in \R$, $H^s(\R)$ denotes the classical
Sobolev space consisting of all functions  $f \in \shs'(\R)$ such that 
$(I-\Delta)^\frac{s}{2} f \in L^2(\R)$.
We introduce   the norm
 $$\Vert f \Vert_{H^s} := \Vert I-\Delta)^\frac{s}{2} f \Vert_{L^2},$$
where $\Vert \cdot \Vert_{L^p}$ is the classical $L^p(\R)$-norm
for $ 1 \le p \le \infty$.
In the sequel, we will often simply denote $H^{-1}(\R)$, 
by  $H^{-1}$ and $L^2(\R)$  by $L^2$.
Furthermore, $W^{r,p}$ denote the classical Sobolev space of order $r \in \N$
in $L^p(\R)$ for $1 \le p \le \infty$. 
\begin{defi} \label{DMultipl}
Given a function $e$ belonging to $L^1_{\rm loc}(\R) \cap  \shs'(\R)$, we say that it is an
{\bf $H^{-1}$-multiplier}, if the map
$ \varphi \mapsto \varphi  e$ 
is continuous from $\shs(\R)$ to $H^{-1}$ 
with respect to the $H^{-1}$-topology on both spaces.
We remark that $  \varphi e$ is always 
 a well-defined Schwartz tempered distribution, 
 whenever $\varphi$ is a fast decreasing test function.
\end{defi}
Of course, any constant function is an  {\bf $H^{-1}$-multiplier}.
In the following lines we give some other sufficient  conditions
on a function $e$ to be an {\bf $H^{-1}$-multiplier}.

\begin{lemma} \label{LMultipl}
Let $e\, : \, \R \to \R$.  If  $e\in W^{1,\infty}$ 
(for instance  if $e \in W^{2,1}$), 
 then $e$ is a $H^{-1}(\R)$-multiplier. 
\end{lemma}

\begin{proof}
For the convenience of the reader we provide a proof. 
We observe that it is enough to show the existence of
a constant $\shc(e)$ such that
\begin{equation}\label{Lmult1} 
\norm{eg}_{H^1} \leq \shc(e)\norm{g}_{H^1},\; \forall\; g \in \Sscr (\R).
\end{equation}
In fact, if \eqref{Lmult1} holds, for every $f\in \shs(\R)$ we have
\begin{align*}
\norm{ef}_{H^{-1}} &= \sup_{\stackrel{g\in\Sscr(\R)}{\norm{g}_{H^1} \leq 1}}  
\int (efg)(x) \d x
\leq \norm{f}_{H^{-1}} \sup_{\stackrel{g\in\Sscr(\R)}{\norm{g}_{H^1} \leq 1}} \norm{eg}_{H^1}\\
&\leq \norm{f}_{H^{-1}} \shc(e),
\end{align*}
which implies that $e$ is a $H^{-1}(\R)$-multiplier.
We verify now \eqref{Lmult1}.\\
For $g \in \Sscr(\R)$  we have
\begin{align*}
\norm{e g}^2_{H^1} &= \int (eg)^2 (x) \d x + \int (eg)^{'2} (x) \d x\\
&\leq \norm{e}_\infty^2 \norm{g}_{L^2}^2 + 2 
\int (e'g)^2 (x) \d x + 2 \int (eg')^2 (x) \d x\\
&\leq \left(\norm{e}_\infty^2 + 2 
\norm{e'}_\infty^2\right) \norm{g}_{L^2}^2 + 2\norm{e}_\infty^2 \norm{g'}_{L^2}^2\\
&\leq \shc(e) \norm{g}^2_{H^1},
\end{align*}
where $\shc(e)=\sqrt 2 \left(\norm{e}^2_\infty + \norm{e'}^2_\infty\right)^{\frac12}$.
\end{proof}

As mentioned in the Introduction, we will consider a fixed filtered
 probability space 
$(\Omega, \shf, P, (\shf_t)_ {t \in [0,T]})$,
where the $ (\shf_t)_ {t \in [0,T]}$ is the canonical filtration
of a standard Brownian motion  $(W^1, \ldots, W^N)$ 
enlarged with the $\sigma$-field generated by $x_0$.  We suppose 
$\shf = \shf_T$. 

Let $(\Omega_1, \shh)$  be a measurable space.
In the sequel, we will also consider 
 another filtered probability space
$(\Omega_0, \shg, Q, (\shg_t)_{t \in [0,T]})$,
where $\Omega_0 = \Omega_1 \times \Omega $, $\shg = \shh \otimes \shf.$

Clearly any random element $Z$ on  $(\Omega, \shf)$
will be implicitly extended to $(\Omega_0, \shg)$
setting $Z(\omega^1,\omega) = Z(\omega)$. It will be for instance 
the case for the above mentioned processes $W^i, i = 1 \ldots N$.

Here we fix some conventions concerning measurability.
Any  topological space $E$ is naturally equipped with its
Borel $\sigma$-algebra $\shb(E)$. 
For instance
$\shb(\R)$ (resp. $\shb([0,T]$) denotes 
the Borel $\sigma$-algebra
of $\R$ (resp. $[0,T]$).

Given any probability space $(\Omega, \shf, P)$,
the $\sigma$-field $\shf$ will always be omitted.
When we will say that a map $T: \Omega \times E \rightarrow \R$
is measurable, we will implicitly suppose that 
the corresponding $\sigma$-algebras are $\shf \otimes \shb(E)$ and $\shb(\R)$.

All the processes on any generic measurable space $(\Omega_2, \shf_2)$
will be considered to be measurable with respect to both
 variables $(t,\omega)$.
In particular any processes on  $\Omega^1 \times \Omega$
is supposed to be measurable with respect to 
$([0,T] \times \Omega^1 \times \Omega, \shb([0,T]) \otimes \shh \otimes \shf)$.

A function $(A, \omega) \mapsto Q(A, \omega)$ 
from $\shh \times \Omega \rightarrow \R_+$ 
is called {\bf random  kernel} (resp.  {\bf random probability kernel})
if for each $\omega \in \Omega$, $Q(\cdot, \omega)$
is a finite  positive (resp. probability) measure and for each $A \in \shh$,
$\omega \mapsto Q(A, \omega)$ is $\shf$-measurable.
The finite measure $Q(\cdot, \omega)$ will also be denoted
by  $Q^\omega$. 
To that random  kernel 
we can associate a specific  finite measure (resp. probability)
denoted by $Q$ on $(\Omega_0, \shg)$
setting $Q(A \times F) = \int_F Q(A, \omega) P(\d \omega) = \int_F Q^\omega(A) P(\d \omega) $,
for $A \in \shh, F\in \shf$.
The probability $Q$ from above 
will be supposed here and below to be associated with a random probability  kernel.


\begin{defi}\label{DSEPS}
If there is a measurable space $(\Omega_1, \shh)$
and a random kernel $Q$ as before, then the probability space
$(\Omega_0, \shg, Q)$ will be called {\bf suitable enlarged probability
space} (of $(\Omega,\shf,P)$).
\end{defi}
As said above, any random variable on $\Omega, \shf)$ will be considered as
a random variable on $\Omega_0 = \Omega_1 \times \Omega$. Then, obviously, $W^1, \ldots, W^N$ 
are independent Brownian motions also 
$(\Omega_0, \shg, Q)$.

Given a local martingale $M$ on any filtered probability space,
 the process $Z:=\she(M)$
denotes its Dol\'eans exponential, which is a local martingale.
 In particular it is the unique
 solution of $\ dZ_t = Z_{t-} dM_t, \quad Z_0 = 1$. 
When $M$ is continuous we have 
 $Z_t = e^{M_t- \halb \langle M \rangle_t}$.

We go on discussing some basic probabilistic tools.
We come back to the notations presented at the beginning of the Introduction,
in particular concerning the random field $\mu$.

Let  $ Z = (Z(s,\xi), s \in [0,T], \xi \in \R)$ be a random field
on $(\Omega, \shf, (\shf_t), P) $
such that $\int_0^T \left (\int_{\R}  \vert Z(s,\xi) \vert \d \xi \right)^2 
\d s < \infty$
a.s. and it is an $L^1(\R)$-valued $(\shf_s)$-progressively measurable process.
 Then the stochastic integral
\begin{equation} \label{DSI}
\int_{[0,t]\times \R} Z(s, \xi)  \mu(\d s, \xi) d \xi := \sum_{i=0}^N \int_0^t 
 \left(\int_\R Z(s,\xi) e^i(\xi) \d \xi\right) \d W^i_s,
 \end{equation}
is well-defined. The consistency of \eqref{DSI} and \eqref{1.2bis} can be seen as follows.\\
Let for a moment $\langle \cdot, \cdot \rangle$ denote the dualization between measures and functions on $\R$, i.e.
$$ \langle \nu, f \rangle := \int_\R f \d \nu, $$
whenever the right-hand side makes sense. Then, for $t \in [0,T]$
$$ \int_{[0,t] \times \R} Z(s, \xi) \mu(\d s, \xi) \d \xi = \sum_{i=0}^N \int_0^t \langle Z(s,\xi) \d \xi, e^i \rangle \d W^i_s $$ 
and
$$ \int \mu(\d s, Y_s(\cdot,\omega)) = \int_0^N \int_0^t \langle 
\delta_{Y_s(\cdot, \omega)},  e^i \rangle 
\d W^i_s,$$
where $\delta_x$ means Dirac measure with mass in $x \in \R$.

We discuss now in which sense the SPDE \eqref{PME}
has to be understood.

 \begin{defi} \label{DSPDE}
 A  random field $X = (X(t, \xi, \omega), 
t \in [0,T]), \xi 
 \in \R, \omega \in \Omega) $ is said to be a solution to \eqref{PME} if
 $P$ a.s. 
we have the following.
\begin{itemize}
\item $X \in C([0,T]; \shs'(\R)) \cap  L^2([0,T]; L^1_{\rm loc} (\R))$.
\item $X$ is an $\shs'(\R))$
-valued $(\shf_t)$-progressively measurable process.
\item 
for any test function $\varphi \in \shs(\R)$ with compact support,
 $ t \in ]0,T]$ 
 we have 
\begin{eqnarray} \label{EDist}
\int_\R X(t,\xi) \varphi(\xi) \d\xi &=& \int_\R x_0(\d\xi) \varphi(\xi)  +
\halb \int_0^t \d s \int_\R 
\eta(s,\xi, \cdot) \varphi''(\xi) \d\xi
 \nonumber\\
& & \\
&+& \int_{[0,t] \times \R} X(s,\xi) \varphi(\xi) \mu(\d s,\xi) \d\xi, 
\nonumber
\end{eqnarray}
where $\eta$ is an $L^1_{\rm loc} (\R) \cap   \shs'(\R)$-valued $(\shf_t)$-progressuvely measurable process
such that for any $\varphi \in \shs(\R)$, we have
$$ \int_0^T \d s \int_\R \vert \eta(s,\xi, \omega) \varphi(\xi) \vert
 \d \xi < \infty,$$
and 
$\eta(s,\xi, \omega) \in \psi(X(s,\xi,\omega)), \quad \d s \d \xi \d P$-a.e.
$(s,\xi, \omega) \in [0,T] \times \R \times \Omega$.
\end{itemize}
\end{defi}
\begin{rem} \label{RDSPDE}
Clearly, if $\psi$ is continuous then $\eta(s,\xi, \cdot) = 
\psi(X(s,\xi,\cdot))$.
\end{rem}

\begin{defi}\label{mylau}
Let  $Y: \Omega_1 \times \Omega \times [0,T] \rightarrow \R$ 
be a measurable process, progressively measurable on $(\Omega_0, \shg, Q,
 (\shg_t)),$ where $(\shg_t)$ is some filtration on $(\Omega_0,\shg,Q)$
such that $W^1, \ldots, W^N$ are $(\shg_t)$-Brownian motions on 
$(\Omega_0,\shg,Q)$.
As we shall see below in Proposition \ref{P28},
 for every $t\in [0,T]$
\begin{equation} \label{e2.1}
E^Q \left( \mathcal{E}_t \left(\int^\cdot_0 \mu (\d s, Y_s) \right) \right) < \infty.
\end{equation}
To $Y$, we will associate its {\bf family of $\mu$-marginal laws}, i.e.
 the family of random kernels ($t \in [0,T]$)
\begin{equation*}
\Gamma_t = \left(\Gamma_t^Y (A,\omega), \; A\in \shb(\R), \; 
\omega \in \Omega\right)
\end{equation*}
defined by
\begin{equation}\tag{E2.2}
\varphi   \mapsto E^{Q^\omega} \left( \varphi (Y_t (\cdot,\omega))
 \she_t \left( \int^\cdot_0 \mu (\d s, Y_s (\cdot , \omega))\right) \right) 
= \int_\R \varphi(r) \Gamma_t^Y(\d r, \omega),
\end{equation}
where $\varphi$ is a generic bounded real Borel function. We will also say that for fixed $t \in [0,T], \; \Gamma_t$ is {\bf the $\mu$- marginal law} of $Y_t$.
\end{defi}
We observe that, taking into account L\'evy's characterization theorem,
the assumption on $W^1,\ldots, W^N$ to be $(\shg_t)$-Brownian motions
can be replaced with  $(\shg_t)$-local martingales.
\begin{rem} \label{R27}
\begin{enumerate}[i)]
\item If $\O$ is a singleton $\{\o_0\}, \; e^i=0,\; 1\leq i\leq N$, the
 $\mu$-marginal laws coincide with the weighted laws
\begin{equation*}
\vp \mapsto E^Q \( \vp (Y_t) \exp \(\int^t_0 e^0 (Y_s) \d s \) \),
\end{equation*}
with $Q=Q^{\o_0}$. In particular if $\mu \equiv0$ then the $\mu$-marginal laws are the classical laws.
\item By \eqref{e2.1}, for any $t\in [0,T]$ , for $P$ almost all
 $\o \in \Omega$,
\begin{equation} \label{ER27}
E^{Q^\o} \( \she_t \(\int^\cdot_0 \mu (\d s, Y_s (\cdot \; , \o))\)\)< \infty.
\end{equation}
\item The function $(t,\o)\mapsto \Gamma_t(A,\o)$ is measurable, for any $A \in
 \shb(\R)$, because $Y$ is a measurable process.
\end{enumerate}
\end{rem}
\begin{prop} \label{P28} 
Consider the situation of Definition \ref{mylau}.
Then we have the following.
\begin{enumerate}[i)]
\item 
The process
 $M_t:=\she_t \( \sum_{i=1}^N \int^\cdot_0 e^i (Y_s)\d W_s^i \)$ is a 
 martingale.
\item The quantity \eqref{e2.1} is bounded by
$ \exp \(T\norm{e^0}_\9 \).$\\
\item $E^Q(M_t^2) \le \exp(3T \sum_{i=1}^N \Vert e^i \Vert_\infty^2),
t \in [0,T]$. Consequently $M$ is a uniformly integrable martingale.  
\item For  $P$-a.e. $\omega \in \Omega$,
\begin{equation*}
\sup_{0 \le t \le T} \norm{\Gamma_t (\cdot,\o)}_{\var} <\9.
\end{equation*}
\end{enumerate}
\end{prop}
\begin{rem} \label{R29}
Proposition \ref{P28} ii) yields in particular that
 $Y$ always admits 
$\mu$-marginal laws.
\end{rem}
\begin{proof}
\begin{enumerate}[i)]
\item The result follows  since the  Novikov condition
\begin{equation*}
E \(\exp\(\frac12 \sum_{i=1}^N \int^t_s e^i (Y_s)^2 \d s\)\)<\9
\end{equation*}
is verified,   because the functions $e^i, \; i=1\dots N$, are bounded.
\item This follows because $E^Q(M_t)=1 \; \forall t\in [0,T]$.
\item $M^2_t$ is equal to
  $N_t \exp\left(3\sum_{i=1}^N \int_0^t (e^i)^2(Y_s) \d s\right), $
where $N$ is a positive martingale with  $N_0 = 1$.
\item For $t\in [0,T]$,
\begin{align*}
\sup_{t\leq T} \norm{\Gamma_t(\cdot\;,\o)}_{\var} &= \sup_{t\leq T} E^{Q^\o} \(M_t \exp \(\int^t_0 e^0 (Y_s) \d s\)\)\\
&\leq \exp \(T\norm{e^0}_\9 \) \sup_{t \le T} E^{Q^\o} \(M_t\).
\end{align*}
Taking the expectation with respect to $P$ it implies
\begin{eqnarray*}
E^P\(\sup_{t\leq T} \norm{\Gamma^Y_t (\cdot\;,\o)}_{\var} \)
&\le& \exp \( T\norm{e^0}_\9 \) E^P\( \sup_{t\leq T} E^{Q^\o} 
\(M_t\)\)\\
&\le &  \exp \(T\norm{e^0}_\9 \) E^P \(E^{Q^\o}\(\sup_{t\leq T} M_t\)\).
\end{eqnarray*}
By the Burkholder-Davis-Gundy (BDG) inequality this is bounded by
\begin{eqnarray*}
3 \exp \(T\norm{e^0}_\9\) E^Q\( \left\langle M \right\rangle^{\frac12}_T\) & \le &
 3 \exp \(T\norm{e^0}_\9\)
 E^Q\(  \left[ \int^T_0 \d s  \sum_{i=1}^N M^2_s
 e^i (Y_s)^2\right]^\halb \) \\
& \le& {\mathrm C}(e,N,T) E^Q \left( \int^T_0 \d s M_s^2 
 \right), 
\end{eqnarray*}
by Jensen's inequality;
 ${\mathrm C}(e,N,T)$ is a constant
 depending on $N,T$ and $e_i,\; i=0\dots N,$.
By Fubini's Theorem and item iii),
 we have
\begin{equation*}
 E^Q \left( \int^T_0 \d s  M_s^2 \right)
\le T\exp(3T \sum^N_{i=1}\Vert e^i \Vert_\infty). 
\end{equation*}
\end{enumerate}
\end{proof}
We go on introducing the concept of weak-strong existence and uniqueness of a stochastic differential equation.
Let $\gamma:[0,T]\times \R \times \O \to \R$ be an $(\shf_t)$-progressively measurable random fields and
 $x_0$ be a  probability on $\shb(\R).$
\begin{defi}\label{DWeakStrong}
\begin{enumerate}
\item[a)] We say that (DSDE)$(\gamma, x_0)$ admits {\bf weak-strong existence}
 if there is a suitable
 extended probability space $(\Omega_0,\shg, Q)$, i.e.
a  measurable space $(\Omega_1,\shh)$, a
 probability kernel $\(Q(\cdot\;,\o), \o\in\O \)$ on $\shh \times \O$, 
two $Q$-a.s. continuous processes $Y,B$ on $(\O_0,\shg)$ 
where $\O_0=\O_1\times \O,\; \shg = 
\shh \otimes \shf$ such that the following holds.
\begin{description}
\item{1)} For almost all $\o$, $Y(\cdot,\o)$ is a (weak) solution to
\begin{equation} \label{DWS}
\begin{cases}
Y_t (\cdot\;,\o)=Y_0 + \int_0^t \gamma (s,Y_s (\cdot\;,\o),\o) \d B_s 
(\cdot\;,\o), \\
\text{Law} (Y_0)=x_0,
\end{cases} 
\end{equation}
with respect to $Q^{\o}$, where $B(\cdot\;,\o)$ is a $Q^\o$-Brownian motion
for almost all $\o$.

\item{2)} We denote $(\shy_t)$ the canonical  filtration
associated with $(Y_s, 0 \le s \le t)$ and
$\shg_t = \shy_t \vee (\{\emptyset, \Omega_1\} \otimes \shf_t)$.
We suppose that $W^1, \ldots, W^N$ is a $(\shg_t)$-martingale 
under $Q$.

\item{3)}
For every $0 \le s \le T$, for every bounded continuous 
$F: C([0,s]) \rightarrow \R$,
the r.v. $\o \mapsto E^{Q^\o}(F(Y_r(\cdot,\omega), r \in [0,s])) $
is $\shf_s$-measurable.

\end{description}
\item[b)] We say that (DSDE)$(\gamma,x_0)$ admits 
{\bf weak-strong uniqueness} if the following holds. 
Consider a measurable space $(\O_1,\shh)$ (resp. $(\wt \O_1, \wt \shh)$), 
a probability kernel $(Q(\cdot\;,\o),\o\in\O)$ (resp.  $(\wt 
Q(\cdot\;,\o),\o\in\O)$), with processes $(Y,B)$ (resp. $(\wt Y, \wt B)$) 
such that 
(\ref{DWS}) 
holds (resp.
(\ref{DWS}) 
  holds with $(\Omega_0,\shg,Q)$ replaced with $(\wt \Omega_0,
\wt \shg_0, \wt Q),$ $\wt Q$ being 
associated with $(\wt Q(\cdot\;, \o)$)).
Moreover we  suppose that item 2. is verified for $Y$ and $\tilde Y$.
\\
Then $(Y,W^1, \ldots, W^N)$ and $(\wt Y, W^1, \ldots, W^N)$ have the same law.
\item[c)] A process $Y$ fulfilling 
items 1) and 2) under (a)
will be called 
{\bf weak-strong solution of} (DSDE)$(\gamma,x_0)$.
\end{enumerate}
\end{defi}
\begin{rem} \label{R2.10}
\begin{description}
\item{a)} Since for almost all $\omega \in \Omega$, 
$B(\cdot,\omega)$ is a Brownian motion under $Q^\omega$,
it is clear that $B$ is a Brownian motion under $Q$,  
which is independent of $\shf_T$, i.e. independent
of $W^1, \ldots, W^N$. \\
Indeed let $G: C([0,T]) \rightarrow \R$ be a continuous
bounded functional, and denote by $\shw$ the Wiener measure.
Let $F$ be a bounded  $\shf_T$-measurable r.v.
 Since for each $\omega$, $B(\cdot,\omega)$ is
a Wiener process with respect to $Q^\omega$,
we get 
\begin{eqnarray*}
 E^Q(F G(B)) &=& \int_\Omega F E^{Q^\omega}(G(B(\cdot,\omega))) \d P(\omega)
= \int_\Omega F(\omega) \d P(\omega) \int_{\Omega_1} G(\omega_1)
 d\shw(\omega_1) \\
 &=& \int_{\Omega_0} F(\omega) \d Q(\omega_0)
 \int_{\Omega_0} G(\omega_1) \d Q(\omega_0).
\end{eqnarray*}
This shows that $(W^1,\ldots,W^N)$ and $B$ are independent.
Taking $F = 1_\Omega$ in previous expression, the equality between the
left-hand side and the third term, shows that 
$B$ is a Brownian motion under $Q$.
\item{b)} Since for any $ 1\le i,j \le N$,
\begin{equation}\label{EIJ}
 [W^i,W^j]_t = \delta_{ij} t,  \ [W^i, B]= 0, \ [B,B]_t = t,
\end{equation}
 L\'evy's characterization theorem, implies that   
$(W^1,\dots,W^N, B)$ is a $Q$-Brownian motion.
\item{c)} By item a) 2) of Definition \ref{DWeakStrong},  
by L\'evy's characterization theorem and 
again by \eqref{EIJ}, 
 it follows that
$W^1, \ldots, W^N$ are $(\shg_t)$-Brownian motions with respect to $Q$.  
\item{d)} An equivalent formulation to 1) 
in item a) of Definition \ref{DWeakStrong} is the following. 
For $P$ a.e.,  $\o \in  \Omega$, $Y(\cdot\;,\o)$ solves the $Q^\o$-martingale problem with respect to the (random) PDE operator
\begin{equation*}
L^\o_t f(\xi) = \frac{1}{2} \gamma^2 (t,\xi,\o) f''(\xi),
\end{equation*}
and initial distribution $x_0$.
\end{description}
\end{rem}
The lemma below shows that, whenever weak-strong uniqueness
holds, then the marginal laws of any weak solution $Y$
are uniquely determined.
\begin{lemma} \label{LMargU}
Let $Y$ (resp. $\tilde Y$) be a process on a suitable 
enlarged probability space $(\Omega_0,\shg, Q)$
(resp.  $(\tilde \Omega_0, \tilde \shg, \tilde Q)$).
Set $W = (W^1, \ldots, W^N)$.
Suppose that the law $(Y,W)$ under $Q$ and
the law of $(\tilde Y, W)$ under $\tilde Q$ are the same.
Then, the $\mu$-marginal laws of $Y$ under $Q$
coincide a.s. with the   $\mu$-marginal laws of $\tilde Y$ under $\tilde Q$. 
\end{lemma}
\begin{proof}  \
Let $0 \le t \le T$.
Using the assumption, we deduce    that
for any  bounded continuous function $f: \R \rightarrow \R$,
and every $F \in \shf_t$, we have
\begin{equation} \label{EMargU}
E^{Q} \left (1_F f(Y_t) \she_t\left(\sum_{i=0}^N 
\int_0^\cdot  e^i(Y_s) dW^i_s \right)\right) = 
E^{\tilde Q} \left (1_F f(\tilde Y_t) \she_t\left(\sum_{i=0}^N 
\int_0^\cdot  e^i(\tilde Y_s) dW^i_s\right)\right). 
\end{equation}
To show this, using
classical regularization properties of It\^o integral,
see e.g. Theorem 2 in \cite{russoSem},
and  uniform integrability arguments, 
 we first  observe that 
$$ \she_t\left(\sum_{i=0}^N 
\int_0^\cdot  e^i(Y_s) dW^i_s\right)$$
is the limit in $L^2(\Omega_0,Q)$ of 
$$ \she_t \left(\sum_{i=0}^N 
\int_0^\cdot  e^i(Y_s) \frac{W^i_{s+\varepsilon} - W^i_s}{\varepsilon} \d s
\right).$$
A similar approximation property arises replacing
$Y$ with $\tilde Y$ and $Q$ with $\tilde Q$.
Then \eqref{EMargU} easily follows.\\ 
To conclude, it will be enough to show the existence
of a countable well-chosen 
 family $(f_j)_{j \in \N}$ of bounded continuous real functions
for which, for $P$ almost all $\omega \in \Omega$, 
for any $j \in \N$, we have $R_j = \tilde R_j$
where 
\begin{eqnarray}\label{ED1}
R_j(\omega) &=& E^{Q^\omega} \left(f_j(Y_t(\cdot,\omega)) 
 \she_t\left(\sum_{i=0}^N 
\int_0^\cdot  e^i(Y_s(\cdot,\omega)) dW^i_s\right)\right) \nonumber\\
&& \\
R_j(\omega) &=& E^{\tilde Q^\omega} \left(f_j(\tilde Y_t(\cdot,\omega)) 
 \she_t\left(\sum_{i=0}^N 
\int_0^\cdot  e^i(\tilde Y_s(\cdot,\omega)) dW^i_s\right)\right) \nonumber.
\end{eqnarray}
This will follow, since applying \eqref{EMargU}, for any
$F \in \shf_t$, we have
$ E^P(1_F R_j) = E^P(1_F \tilde R_j)$.
\end{proof}

\begin{prop} \label{PB1}
Let  $Y$ be a process  as in Definition \ref{DWeakStrong} a). 
We have the following.
\begin{enumerate}
\item $Y$ is a $(\shg_t)$-martingale
on the product space $(\Omega_0,\shg, Q)$.
\item $[Y, W^i] = 0, \ \forall 1 \le i \le N$. 
\end{enumerate}
\end{prop}
\begin{proof} Let $0 \le s < t \le T$, $F_s \in \shf_s$ and
 $G: C([0,s]) \rightarrow \R$ be continuous and bounded. 
We will  prove below  that, for $1 \le i \le N+1$,
setting $W^{N+1}_t = 1$, for all $t \ge 0$,
\begin{equation} \label{EB1}
E^Q (Y_t W^i_t  G(Y_r, r \le s) 1_{F_s} ) = E^Q (Y_s W^i_s 1_{F_s} 
 G(Y_r, r \le s)).
\end{equation}
Then \eqref{EB1} with $i= N+1$ shows  item 1.
Considering \eqref{EB1} with $1 \le i \le N$,
shows that $Y W^i$ is a $(\shg_t)$-martingale, which shows
item 2.
Therefore, it remains to show \eqref{EB1}.\\
The left-hand side of that equality gives
\begin{eqnarray*}
\int_\Omega \d P(\omega) &&  W^i_t(\omega) 1_{F_s}(\omega)  E^{Q^\omega} 
\left(Y_t(\cdot,\omega)  G(Y_r(\cdot,\omega), r \le s) \right) \\
&=&
\int_\Omega \d P(\omega) 1_{F_s}(\omega) W^i_t(\omega) 
 E^{Q^\omega} \left(Y_s(\cdot,\omega) G(Y_r(\cdot,\omega), r \le s)\right),
\end{eqnarray*}
because $Y(\cdot, \omega)$ is a $Q^\omega$-martingale for $P$-almost all
 $\omega$.
To obtain the right-hand side of   \eqref{EB1}
 it is enough to remember that  $W^i$ are 
$(\shg_t)$-martingales and that item a) 3) in Definition \ref{DWeakStrong}
holds.
This concludes the proof of Proposition \ref{PB1}.
\end{proof}

\section{The concept of double stochastic non-linear diffusion.}
\label{S3}

\setcounter{equation}{0}

We come back to the notations and conventions of the introduction and of Section \ref{S2}.
Let $x_0$ be a probability on $\R$.
\begin{defi}\label{DDoubleStoch}
\begin{enumerate}
\item[1)] We say that the double stochastic non-linear diffusion (DSNLD)
 driven by $\Phi$ (on the space $(\O,\shf,P)$ with initial condition $x_0$,
 related to the random field $\mu$ (shortly (DSNLD)$(\Phi, \mu, x_0)$)
 admits {\bf weak existence} if there is 
a measurable  random field $X : [0,T]
 \times \R \times \O \to \R$ with the following properties.
\subitem a) 
The problem (DSDE)$(\gamma,x_0)$  with
 $\gamma = \chi$ for some measurable
$\chi:[0,T] \times \R \times \Omega \rightarrow \R$
such that
 $\chi(t,\xi,\o)\in \Phi (X(t,\xi,\o))\d t \d\xi \d P$ a.e.
admits weak-strong existence.
\subitem b) $X = X(t,\xi, \cdot) \d \xi, t \in ]0,T]$,
 is the family of $\mu$-marginal laws of $Y$. In other words
$X$ constitutes the densities of those $\mu$-marginal laws.
\item[2)] A couple $(Y,X)$, such that $Y$ is a (weak-strong) solution to the \\
(DSDE)$(\chi,x_0)$, with $\chi$ 
as in item 1) a), which also fulfills 1) b),  is 
called {\bf weak solution} to the (DSNLD)$(\Phi,\mu,x_0)$. 
 $Y$ is also called double stochastic representation of the random field $X$.
\item[3)]
Suppose that, given two measurable 
 random fields $X^i: [0,T] \times \R \times \O\to \R,
 i = 1,2$ on $(\O, \shf, P, (\shf_t))$,
  and $Y^i$, on extended probability space $(\Omega_0^i, Q^i), i=1,2$,
such that $(Y^i, X^i)$ is a
  weak-strong solution  of 
(DSDE)$(\chi^i,x_0), i =1,2$ where
 $\chi^i \in \Phi (X^i)\  \d t \d\xi \d P$ a.e.,
we  always have that $(Y^1, W^1, \ldots, W^N)$ 
and   $(Y^2, W^1, \ldots, W^N)$ have the same law. Then 
we say that the (DSNLD)$(\Phi,\mu,x_0)$ admits
 {\bf weak uniqueness}.
\end{enumerate}
\end{defi}
\begin{rem} \label{RRRR}
If (DSNLD)$(\Phi,\mu,x_0)$ admits {\bf weak uniqueness} then 
the $\mu$-marginal laws of $Y$
 is uniquely determined, $P$-a.s., see Lemma \ref{LMargU}.
\end{rem}
The first connection between \eqref{PME} 
with $\psi(u) = \Phi^2(u) u$
and  (DNSLD)$(\Phi,\mu,x_0)$ is the following.
\begin{theo} \label{T32}
Let $(Y,X)$ be a solution of (DSNLD)$(\Phi,\mu,x_0)$. 
Then $X$
 is a solution to the SPDE \eqref{PME}.
\end{theo}
\begin{proof}
Let $B$ denote the Brownian motion associated to $Y$ as 
a solution   to  (DSDE)$(\chi,x_0)$,
mentioned in item a)1)  of Definition \ref{DDoubleStoch},
with $\gamma = \chi$.
For $t \in [0,T]$, we set
\begin{eqnarray*}
Z_t &=& \she_t\(\int^\cdot_0 \mu (\d s,Y_s)\), \\
M_t &=& Z_t \exp \left(-\int_0^t e^0(Y_s) \d s \right), \ t\in [0,T]. 
\end{eqnarray*}
\begin{enumerate} 
\item We first prove that the first item of Definition \ref{DSPDE}
is verified.
By Proposition \ref{P28}, $(M_t, t \in [0,T])$ is a uniformly
integrable martingale. Consequently $t \mapsto Z_t$ is continuous
in $L^1(\Omega, Q)$. On the other hand the process $Y$ is continuous. 
This implies that $P$ a.e. $\omega \in \Omega$, $X \in C([0,T]; \shm(\R))$,
where $\shm(\R)$ is equipped with the weak topology.
This  implies that   $X \in C([0,T]; \shs'(\R))$.
Furthermore, for $P$ a.e. $\omega \in \Omega$, and $t\in ]0,T]$,
$X(t,\cdot, \omega) \in L^1(\R)$ and 
$\int_\R X(t,\xi,\omega) \d \xi = \Vert \Gamma(t,\cdot,\omega)\Vert_{\rm var}$.
By item iv) of Proposition \ref{P28}, it follows that $P$-a.s.
$$ X \in L^\infty([0,T];L^1(\R)) \subset L^2([0,T];L^1_{\rm loc}(\R)).$$ 
\item We prove now the validity of the third item of 
Definition \ref{DSPDE}. 
Let $\varphi\in \shs(\R)$ with compact support.
For simplicity of the formulation we suppose here $\psi$ to be single-valued.
Taking into account Proposition \ref{PB1}, we apply It\^o's formula to get
\begin{align*}
&\varphi(Y_t) Z_t = \varphi(Y_0)+\int^t_0\varphi' (Y_s) Z_s \d Y_s\\
&+\int^t_0 \varphi(Y_s) Z_s\(\mu (\d s,Y_s) -\frac12 \sum^N_{i=1} (e^i(Y_s))^2 \d s \)\\
&+\frac12 \int^t_0 \varphi''(Y_s) \Phi^2 (X(s,Y_s)) Z_s \d s\\
&+\frac12 \int^t_0 \varphi(Y_s) Z_s \left( \sum_{i=1}^N (e^i (Y_s))^2 \right) \d s.
\end{align*}
Indeed we remark that
\begin{equation*}
\int^t_0 \varphi' (Y_s) \d [Z,Y]_s = 0,
\end{equation*}
because
\begin{equation*}
[Z,Y]_t = \sum_{i=1}^N \int^t_0 e^i (Y_s) Z_s  \d [W^i,Y]_s = 0;
\end{equation*}
in fact $[W^i,Y]=0$ by Proposition \ref{PB1}.
So
\begin{align*}
\varphi (Y_t) Z_t &= \varphi(Y_0) + \int^t_0 \varphi' (Y_s) Z_s \Phi (X(s,Y_s)) \d B_s\\
&+ \int^t_0 \varphi (Y_s) Z_s \mu (\d s, Y_s)\\
&+\frac12 \int^t_0 \varphi'' (Y_s) \Phi^2 (X(s,Y_s)) Z_s \d s.
\end{align*}
Taking the expectation with respect to $Q^\o$ we get $\pas$,
\begin{align*}
\int_\R \d \xi \varphi (\xi) X(t,\xi) &= \int_\R \varphi(\xi)x_0(\d\xi) +
\sum_{i=0}^N \int^t_0 \d W_s^i
 \( \int_\R \d \xi \varphi(\xi) e^i (\xi) X(s,\xi)\)\\
&+ \frac12 \int_0^t \d s \int_\R \d \xi \varphi'' (\xi) \Phi^2 (X(s,\xi)) X(s,\xi),
\end{align*}
which implies the result. Indeed, in the previous equality,  we have used  
the lemma below.
\end{enumerate}
\end{proof}
\begin{lemma} \label{Lmulaw}
Let $ 1 \le i \le N$.
For $P$ a.e. $\omega \in \Omega$, we have
$$ E^{Q^\omega}\left( \int^t_0 \varphi (Y_s) Z_s e^i(Y_s) \d W^i_s \right)
(\cdot,\omega)  
= \int_0^t \d W^i_s(\omega) \int_\R \varphi(\xi) e^i(\xi) X(s,\xi,\omega) \d \xi.
$$
\end{lemma}
\begin{proof} 
Since $ E^Q (\int_0^T (\varphi(Y_s) Z_s)^2 \d s) < \infty$, again the usual
regularization properties of the It\^o integral (see e.g.
  Theorem 2, \cite{russoSem}), 
give
$$ \lim_{\varepsilon \rightarrow 0} 
E^Q \left \vert \int_0^T \frac{W_{s+\varepsilon} - W_s}{\varepsilon}
\varphi(Y_s) Z_s e^i(Y_s) \d s
- \int_0^T   \varphi(Y_s) Z_s e^i(Y_s) \d W^i_s \right \vert = 0.$$
This implies the existence of a sequence $(\varepsilon_\ell)$ 
such that   $P$ a.e. $\omega \in \Omega$,  
\begin{eqnarray*}
 \lim_{\ell \rightarrow \infty} 
E^{Q^\omega} &&\left \vert \int_0^T 
\frac{W^i_{s+\varepsilon_\ell}(\omega)
 - W^i_s(\omega)} 
{\varepsilon_\ell}
\varphi(Y_s(\cdot,\Omega)) Z_s(\cdot,\omega) e^i(Y_s(\cdot,\omega)) \d s 
\right.\\
&-& \left.  \int_0^T  \varphi(Y_s(\cdot, \omega)) Z_s(\cdot,\omega)
 e^i(Y_s(\cdot,\omega)) \d W^i_s(\omega) \right \vert = 0. 
\end{eqnarray*}
So  $P$-a.e $\omega \in \Omega$, 
\begin{eqnarray} \label{C2}
\lim_{\ell \rightarrow \infty} 
E^{Q^\omega} &&  \left (\int_0^T \frac{W^i_{s+\varepsilon_\ell}(\o) -
 W^i_s(\omega)}
{\varepsilon_\ell}
\varphi(Y_s(\cdot,\o)) Z_s(\cdot,\o) e^i(Y_s(\cdot,\o)) \d s \right. 
\nonumber \\
&&\\
&-& \left. \int_0^T   \varphi(Y_s(\cdot,\o)) Z_s(\cdot,\o)
 e^i(Y_s(\cdot,\o)) \d W^i_s (\o) \right) = 0. \nonumber
\end{eqnarray}
The left-hand side \eqref{C2}, by Fubini's, gives
\begin{eqnarray*} 
&&  \int_0^T \frac{W^i_{s+\varepsilon_\ell}(\o) - W^i_s(\o)}{\varepsilon_l}
E^{Q^\omega} (\varphi(Y_s(\cdot,\o)) Z_s(\cdot,\o) e^i(Y_s(\cdot,\o)))\\
&=&   \int_0^T \frac{W^i_{s+\varepsilon_\ell}(\o) -
 W^i_s(\o)}{\varepsilon_\ell}
\int  \varphi(\xi) e^i(\xi) X(s, \xi,\o) \d \xi \\
&=& \int_0^T \d W^i_s \int  \varphi(\xi) e^i(\xi) X(s, \xi,\o) \d \xi,
\end{eqnarray*}
again by Theorem 2 of \cite{russoSem}.

\end{proof}

\section{The densities of the $\mu$-marginal laws}

\label{S4}

\setcounter{equation}{0}

This section constitutes an important step towards the double probabilistic representation
of a solution to \eqref{PME}, when $\psi$ is non-degenerate. 
Let $x_0$ be a fixed probability on $\R$.
We remind that a process $Y$ (on a suitable enlarged probability space $(\O_0, \shg, Q)$), which 
is a weak solution to the (DSNLD)$(\Phi,\mu,x_0)$, is in particular a weak-strong solution of a (DSDE)$(\gamma,x_0)$ 
where $\gamma:[0,T] \times \R \times \O \rightarrow \R$ is some suitable 
progressively measurable random field on $(\O,\shf,P)$.
The aim of this section is twofold.
\begin{enumerate}
\item[A)] To show that whenever $\gamma$ is a.s. bounded and non-degenerate, (DSDE)$(\gamma,x_0)$ admit weak-strong existence and uniqueness.
\item[B)] The marginal $\mu$-laws of the solution to (DSDE)$(\gamma,x_0)$ admits a density for $P \; \o$ a.s.
\item[A)] We start discussing well-posedness.
\end{enumerate}
\begin{prop} \label{P4.1}
We suppose the existence of random variables $A_1, A_2$ such that
\begin{equation} \label{he4.2}
0< A_1 (\o) \leq \gamma (t,\xi,\o) \leq A_2 (\o) \quad \pas.
\end{equation}
Then (DSDE)$(\gamma,x_0)$ admits weak-strong existence and uniqueness.
\end{prop}
\begin{proof}
\textit{Uniqueness.} This is the easy part. Let $Y$ and $\wt Y$ be two solutions. Then for $\o$ outside a $P$-null set $N_0, Y(\cdot\;, \o)$ and $\wt Y (\cdot\;, \o)$ are solutions to the same one-dimensional classical SDE with measurable bounded and non-degenerate coefficients. Then, by Exercise 7.3.3
of  \cite{sv} the law of $Y(\cdot\;,\o)$ equals the law of $\wt Y(\cdot\;,\o)$. Then obviously the law of $Y$ equals the law of $\wt Y$.

\medskip
\textit{Existence.} This point is more delicate. In fact one needs 
to solve the random SDE for $P$ almost all $\o$  but in such
a way that the solution produce bimeasurable processes $Y$ and $B$.

First we regularize the coefficient $\gamma$. Let $\phi$ be a mollifier
 with compact support; we set
\begin{equation*}
\phi_n (x) = n\phi (nx), \; x\in \R \; , \; n\in \N.
\end{equation*}
We consider the random fields $\gamma_n : [0,T]
 \times \R \times \O \to \R$ by $\gamma_n (t,x,\o) := \int_\R 
\gamma (t,x-y,\o) \phi_n (y) \d y$. 
\\
Let $(\wt \O_1, \wt \shh_1,\wt P)$ be a probability space where
 we can construct a random variable $Y_0$ distributed according to $x_0$ 
and an independent Brownian motion $B$.

In this way on $(\wt \O_1 \times \O, \wt \shh_1 \otimes \shf, \wt P \otimes P)$ we dispose of a random variable $Y_0$ and a Brownian motion independent of
 $\{\phi,\O\} \otimes \shf$. By usual fixed point techniques, it is possible to exhibit a (strong) 
solution of (DSDE)$(\gamma_n,x_0)$ on the 
overmentioned product probability space.
We can show that there is a unique solution  $Y = Y^n$ of
\begin{equation}\label{e4.1bs}
Y_t  = Y_0 + \int^t_0 \gamma_n (s,Y_s,\cdot) \d B_s.
\end{equation}
In fact, the maps 
\begin{equation*}
\Gamma_n : Z \mapsto \int_0^\cdot \gamma_n (s,Z_s,\o) \d B_s + Y_0,
\end{equation*}
where $\Gamma_n : L^2(\wt \O_1 \times \O; \; \wt P \otimes P) \to L^2 (\wt \O_1 \times \O, \wt P \otimes P)$ are Lipschitz; by usual Picard fixed point arguments one can show the existence of a unique solution $Z=Z^n$ in $L^2(\wt \O_1 \times \O; \; \wt P \otimes P)$.
We observe that, by usual regularization arguments for It\^o integral
 as in Lemma \ref{Lmulaw},
for $\omega$-a.s.,
$Y(\cdot,\omega)$ solves for $P$ a.e. $\omega \in \Omega$, equation
\begin{equation}\label{e4.1}
Y_t (\cdot\;,\o) = Y_0 + \int^t_0 \gamma_n (s,Y_s (\cdot\;,\o),\o) \d B_s,
\end{equation}
on $(\wt \O_1, \wt \shh_1,\wt P)$.
We consider now the measurable space $\O_0 = \O_1 \times \O$, where 
$\O_1 = C([0,T], \; \R) \times \R,$ 
equipped with product $\sigma$-field $\shg = \shb(\Omega_1) \otimes \shf$.
On that measurable space, we introduce  the
 probability measures $Q_n$ where $Q_n = \int_\O \d P(\o) Q_n (\cdot, \o)$ and
 $Q_n (\cdot, \o)$ being the law of $Y^n(\cdot\;,\o)$ for almost all fixed $\o$.
\\
We set $ Y_t(\o_1,\o)=\o_1(t)  $,
where $\omega_1(t) = \omega_1^0(t) + a$, if $\omega^1 = (\omega_1^0,a)$.
We denote by $( \shy_t, \; t\in[0,T])$ (resp. $(\shy^1_t)$)
 the canonical 
 filtration associated with $Y$ on $\Omega_0$ (resp. $\Omega_1$).
 The next step will be the following.

 \begin{lemma} \label{Lconv}
For almost all $\o  \ dP$ a.s.
$Q_n (\o,\;\cdot)$ converges weakly to $Q (\o, \cdot)$, where under
 $Q(\cdot, \o), \; Y(\cdot\;, \o)$ solves the SDE
\begin{equation}\label{e4.2}
Y_t(\cdot\;, \o) = Y_{0} + \int^t_0 \gamma (s,Y_s(\cdot\;, \o),\o) 
\d B_s(\cdot,\omega),
\end{equation}
where $B(\cdot, \omega)$ is an $(\shy^1_t)$-Brownian motion 
on $\Omega_1$.
\end{lemma}
\begin{proof} \
It follows directly from Proposition \ref{A4} of the Appendix.
\end{proof}
\begin{rem} \label{RConv1}
\begin{enumerate}
\item[1)] Since $Q_n (\cdot, \o)$ converges weakly to $Q (\cdot, \o)$,
 $\o \ dP$ a.s., then the limit (up to an obvious modification) is a measurable random kernel.
\item[2)] This also implies that $Y_n(\cdot, \o)$ converges stably to
 $Q (\cdot, \o)$. 
For details about the stable convergence the reader can consult 
\cite[section VIII 5. c]{jacod}.
\end{enumerate}
\end{rem}
The considerations above allow to conclude the proof of Proposition \ref{P4.1}
By Lemma \ref{Lconv}, $Q^\omega  = Q(\cdot,\omega)$ is
a random kernel, being a limit of random kernels.
Let us consider  the associated probability measure on the suitable
enlarged probability space $(\Omega_0,\shg,Q)$.
 We observe that $Y$ on $(\O_0,\shg)$
 is obviously measurable, because it is the canonical process $Y(\o_1,\o)=\o_1$.
Setting
\begin{equation*}
B_t(\cdot, \o)  = \int_0^t \frac{\d Y_s}{\gamma (s, Y_s,\o)} , 
\end{equation*}
we get $[B]_t(\cdot, \o) =t$ under $Q(\cdot\;, \o)$, so,
by L\'evy characterization theorem,  it is a Brownian motion.
 Moreover $B$ is bimeasurable.
The last point to check is that 
$W^1,\ldots, W^N$ are $(\shg_t)$-martingales,
where $\shg_t = (\shf_t \otimes \{\emptyset, \Omega_1\})
  \vee \shy_t, \ 0 \le t \le T$.
 \\
Indeed, we justify this immediately. Condider $0 < s \le t \le T$.
Taking into account monotone class arguments, given $F \in \shf_s$,
$G \in \shy^1_s$, $1 \le i \le N $, 
it is enough to prove that
\begin{equation} \label{ECondExp}
 E^Q ( F G W^i_t) = E^Q(FG W^i_s).
\end{equation}
We first observe that the r.v. 
$\omega \mapsto E^{Q^\omega}(G)$
is $\shf_s$-measurable. 
This happens because $Y$ is, under $Q^\omega$, 
a martingale with quadratic variation \\
$\left(\int_0^t \gamma^2_(s, Y_s(\cdot, \omega),\omega) \d s, 0 \le t \le 
T\right)$, 
i.e. with (random) coefficient which is $(\shf_t)$-progressively measurable. \\
Consequently, also using the fact that
 $W^i$ is an $(\shf_t)$-martingale and that $E^{Q^{\omega}} (G)$
is $\shf_s$-measurable by item a) 3) of Definition \ref{DWeakStrong},
 the left-hand side of previous equality gives
$$
 E^P ( F W^i_t  E^{Q^{\omega}} (G)) =
  E^P ( F W^i_s  E^{Q^{\omega}} (G)), 
$$
which constitutes the right-hand side of
\eqref{ECondExp}. 
This concludes the proof of the proposition.
\end{proof}
We go on now with step B) of the beginning of Section \ref{S4}.
\begin{prop} \label{P4.2}
We suppose the existence of r.v. $A_1,A_2$ such that
\begin{equation} \label{4.2bis}
0 < A_1(\o)\leq \gamma (t,\xi,\o)\leq A_2(\o), \forall
(t,\xi) \in [0,T] \times \R, \quad \text{ a.s.}
\end{equation}
Let $Y$ be a weak-strong solution to (DSDE)$(\gamma,x_0)$
and we denote by $(\nu_t(dy,\cdot), t \in [0,T])$, the $\mu$-marginal laws of process $Y$. 
\begin{enumerate}
\item  There is a  measurable function $q: [0,T] \times \R \times \Omega \rightarrow \R_+$
such that  $\d t \d P $ a.e., 
$\nu_t(\d y,\cdot) = q_t(y,\cdot) \d y$.
In other words the  $\mu$-marginal laws admit densities.
\item
\begin{equation}  \label{e4.3}
\int_{[0,T]  \times \R} q^2_t (y,\; \cdot) \d t \d y < \9 \quad \pas
\end{equation}
\item $q$ is an $L^2(\R)$-valued progressively measurable process.  
\end{enumerate}
\end{prop}
\begin{proof}
By 3) of Definition \ref{DWeakStrong}, the $\mu$-marginal laws constitute
an $\shs'(\R)$-valued progressively measurable process.
Consequently  3. holds if 1. and 2. hold.

Let
\begin{equation*}
B_t := \int_0^t \frac{\d Y_s}{\gamma(s,Y_s,\o)}.
\end{equation*}
We denote again $Q^\o:=Q(\cdot\;, \o)$ according to Definition \ref{DWeakStrong}, $\omega \in \Omega$.\\
Let $\o \in\O$ be fixed. Let $\varphi:[0,T] \times \R\to \R$ be a continuous function with 
compact support.
We need to evaluate
\begin{equation}\label{e4.4}
E^{Q^\o} \( \int_0^T \varphi (s,Y_s) Z_s \d s \),
\end{equation}
where $Z_s = \exp \(\int_0^s e^0 (Y_r)\d r \).\; M_s$ and
 $M_s = \she_s \(\sum^N_{i=1} \int^\cdot_0 e^i (Y_r)\d W_r^i \).$
$M_s$ is smaller or equal than
\begin{align*}
& \exp \( \sum^N_{j=1} \int^s_0 e^j \(Y_r\) \d W^j_r\)\\
= & \exp \( \sum^N_{j=1} \left\{ W_s^j e^j(Y_s) -
 \int^s_0 W^j_r (e^j)'(Y_r) 
\d Y_r \right \} \) \\
&\times \exp \( -\frac{1}{2} \int_0^s \sum^N_{j=1} \left\{ W_r^j (e^j)''(Y_r)  
\gamma^2(r, Y_r, \cdot)  \d r 
- \frac{1}{2} W^j_r (e^j)'(Y_r) \gamma^2(r, Y_r, \cdot) \d r 
\right  \} \),
\end{align*}
taking into account the fact that $[Y, W^j] = 0$ for any $1 \le j \le n$,
by Proposition \ref{PB1}.\\
Denoting $\Vert g \Vert_\infty : = \sup_{t \in [0,T]} \vert g(t) \vert,$
for a function $g:[0,T] \rightarrow \R$,
\eqref{e4.4} is  smaller or equal than
$$ \exp \( \sum^N_{j=1}  \Vert W^j \Vert_\infty
 (\norm{e^j}_\infty + \frac{T}{2}  \norm{(e^j)''}_\infty A_2^2(\omega)) \)
 \exp \( - \int_0^s \left[ \sum^N_{j=1}  W^j_r (e^j)'
 (Y_r)
\gamma (r,Y_r,\o)\right]\d B_r\).$$
So \eqref{e4.4} is bounded by
\begin{equation}\label{e4.5}
\varrho(\o) E^{Q^\o} \(\int^T_0 \vert \varphi \vert (s,Y_s) R_s \d s\),
\end{equation}
where
\begin{eqnarray*}
\varrho(\o) &=& \exp \left(T\norm{e_0}_\9 + \sum^N_{i=1} \norm{W^i}_\9 
\norm{e^i}_\9 \right.\\
 &+&  \left. T \frac{A_2^2(\omega)}{2} \sum^N_{i=1} 
(\norm{W^i}^2_\9 \norm{(e^i)'}^2_\9 + \Vert W^i \Vert_\infty \Vert 
(e^i)''\Vert_\infty) \right)
\end{eqnarray*}
and $R$ is the $Q^\o$-exponential martingale
\begin{eqnarray*}
R_t (\; \cdot \;, \o) = \exp \bigl( &-&\int^t_0 \delta (r, \; \cdot \;,\o) 
\d B_r  \\
  &-& \frac12 \int^t_0 \delta^2 (r,\;\cdot\;,\o)\d r \bigr).
\end{eqnarray*}
where
\begin{equation*}
\delta (r,\;\cdot\;,\o) = \sum^N_{j=1} W_r^j (e^j)' \(Y_r(\; \cdot \; , \o)\)
 \gamma \(r,Y_r(\;\cdot\;,\o), \o \).
\end{equation*}
So there is a random (depending on $(\O,\shf)$) constant
\begin{equation}\label{e4.6}
\varrho_1 (\o) := {\rm const} \(T,W^j,\norm{e^j}_\9,\norm{e^{j'}}_\9,
\norm{e^{j''}}_\9,
 \; 1\leq j \leq N,\; A_2(\o)\),
\end{equation}
so that \eqref{e4.5} is smaller than
\begin{equation}\label{e4.7}
\varrho_1(\o) E^{Q^\o} \( \int^T_0 
\vert \varphi(s,Y_s(\;\cdot\;,\o))\vert \d s R_T (\;\cdot\;,\o) \).
\end{equation}
By Girsanov theorem,
\begin{equation*}
\wt B_t(\cdot,\o) = B_t(\cdot, \o) + \int_0^t \delta(r,\;\cdot\;,\o)\d r
\end{equation*}
is a $\wt Q^\o$-Brownian motion with
\begin{equation*}
\d \wt Q^\o = R_T (\;\cdot\;,\o) \d Q^\o.
\end{equation*}
At this point, the expectation in \eqref{e4.7} gives 
\begin{equation}\label{e4.8}
E^{\wt Q^\o} \(\int^T_0  \vert \varphi \vert(s,Y_s (\;\cdot\;,\o)) \d s \),
\end{equation}
where
\begin{align*}
Y_t (\;\cdot\;,\o) &= Y_0 + \int^t_s \gamma (s,Y_s(\;\cdot\;,\o),\o)
 \d \wt B_s\\
&- \int_0^t \gamma (s,Y_s (\;\cdot\;,\o), \o) \delta (s, \;\cdot\;,\o) \d s.
\end{align*}
For fixed $\o\in\O$, $\delta$ is bounded by a random constant $\varrho_2(\o)$ of the type \eqref{e4.6}.
Moreover we keep in mind assumption \eqref{e4.1} on $\gamma$.
By Exercise 7.3.3 of \cite{sv}, \eqref{e4.8} is bounded by
\begin{equation*}
\varrho_3(\o)\norm{\varphi}_{L^2 ([0,T]\times \R)}.
\end{equation*}
where $\varrho_3(\o)$ again depends on the same quantities as in \eqref{e4.6} and $\Phi$.
So for $\o \pas$,
 the map $\varphi \mapsto E^{Q^\o} \(\int_0^T \varphi
 (s,Y_s (\;\cdot\;,\o)) \wt M_s (\;\cdot\;,\o) \d s\right)$ prolongates to $L^2([0,T]\times \R)$. Using Riesz theorem
 it is not difficult to show the existence of an $L^2([0,T]\times\R)$
function $(s,y)\mapsto q_s (y,\o)$ which constitutes indeed the density of the family of the $\mu$-marginal laws.
\end{proof}

\section{On the uniqueness of a Fokker-Planck type SPDE}
\label{S5}

\setcounter{equation}{0}

The theorem below plays the analogous role as Theorem 3.8 in \cite{BRR1}
or Theorem 3.1 in \cite{BCR2}.
It has an interest in itself since it is
 a Fokker-Planck SPDE with possibly degenerate measurable coefficients.

\begin{theo} \label{T51}
Let $z^1, z^2$ be two measurable random fields belonging $\o$ a.s. to 
$C([0,T],\shs'(\R))$ such that $z: ]0,T]\times \O \to \shm (\R)$.
 Let $a:[0,T]\times\R\times\O\to\R_+$ be a bounded measurable random field
such that, for any $t \in [0,T]$, $a(t, \cdot)$ is
$\shb([0,t]) \otimes \shb(\R) \otimes  \shf_t$-measurable. 
We suppose moreover the following.
\begin{enumerate}
\item[i)] $z=z^1-z^2 \in L^2 ([0,T]\times \R)$ a.s.
\item[ii)] $t \mapsto z(t,\cdot)$ is $(\shf_t)$-progressively measurable
$\shs'(\R)$-valued process.
\item[iii)] $z^1,z^2$ are solutions to
\begin{align}\label{e5.1}
\begin{cases} 
\partial_t z(t,\xi) = \partial^2_{\xi\xi} ((a z)(t,\xi)) + z(t,\xi) 
\mu (\d t, \xi),\\
z(0,\;\cdot\;) = z_0,
\end{cases}
\end{align}
where $z_0$ is some distribution in $\shs'(\R)$.
\end{enumerate}
Then $z^1 \equiv z^2$.
\end{theo}

\begin{rem} \label{R5.2}
\begin{enumerate}
 \item[a)] By solution of equation \eqref{e5.1} we intend,
 as expected, the following: for every $\varphi \in \shs(\R), \ \forall t \in [0,T]$,
\begin{align}\label{e5.2}
\int_\R \varphi(x) z(t,\d \xi) &= \left< x_0, \varphi \right>  + \int^t_0 \d s \int_\R a (s,\xi)\varphi''(\xi) z(s,\d \xi)\\
&+\int_{[0,t]\times \R} \mu (\d s, \xi) z(s,\d \xi) \varphi(\xi) \quad \text{a.s.}
\end{align}
\item[b)] Since $z(\cdot,\;\omega)$ is $\o$ a.s. in 
$L^2([0,T];L^2(\R)) \; \subset L^2([0,T];H^{-1}(\R))$,
 then $\int^t_0\mu(\d s, \; \cdot\;) z(s, \cdot)$ 
belongs  $\o$ a.s. to $C([0,T];H^{-1}(\R))$.
On the other hand  $\int^t_0 (az)''(s, \cdot)\d s$ can be seen as Bochner integral in $H^{-2}(\R)$ and
so $t \mapsto  \int^t_0\mu(\d s, \; \cdot\;) z(s, \cdot)$ belongs to
$ C([0,T];H^{-2}(\R))$ $\o$ a.s.
In particular any solutions $z^1,z^2$ to \eqref{e5.1} are such that 
$z = z^1 - z^2$
admits a modification whose paths belong (a.s.)
to  $C([0,T];H^{-2}(\R)) \cap L^2([0,T];L^2(\R))$. Since $z^i, i =1,2$, are continuous with values in $\shs'(\R)$,
then their difference is indistinguishable  with the mentioned modification.
\\
Consequently for $\o$ a.s. $z(t,\cdot) \in C([0,T];H^{-2}(\R))$.
 Then, outside a $P$-null set $N_0$, for $\o\in N_0$ 
we have (in $\shs'(\R)$ and  $H^{-2}(\R))$ )
\begin{equation}\label{e5.3}
z(t,\cdot) = \int^t_0 (az)''(s,\cdot)\d s+\int_0^t \mu (\d s,\cdot)z(s,\cdot).
\end{equation}
\item[c)] By assumption i),
possibly enlarging the $P$-null set $N_0$ we get the following.
For $\o\notin N_0$,
for almost all $t\in ]0,T]$,
$ \( \int^t_0 (az) (s,\cdot) \d s\right)''\in H^{-1}(\R)$ and so 
$\int^t_0 (az) (s,\;\cdot\;) \d s\in H^1\; \d t$ a.e.
\end{enumerate}
\end{rem} 

\begin{proof}  [Proof of Theorem \ref{T51}]
We fix the null set $N_0$ and so $\o$ will always lie outside $N_0$ related to Remark \ref{R5.2} c).
Let $\phi$ be a mollifier with compact support and $\phi_\varepsilon = \frac{1}{\varepsilon} \phi (\frac{\cdot}{\varepsilon})$ be a  generalized 
sequence of mollifiers converging to the Dirac delta function. We set
$$
g_\varepsilon(t) = \norm{z_\varepsilon (t)}^2_{H^{-1}}
= \int_\R \d \xi z_\varepsilon (t,\xi) ((I-\Delta)^{-1} z_\varepsilon )(t,\xi), 
$$
where $z_\varepsilon(t,\xi) = \int_\R \phi_\varepsilon (\xi-y) z(t,\d y) $.
Since $t \mapsto z(t,\;\cdot\;)$ is continuous in $H^{-2}(\R)$, then 
$t\mapsto z_\varepsilon (t,\;\cdot\;)$ is continuous in $L^2(\R)$ and so also in $H^{-1}(\R)$.
 We look at the equation fulfilled by $z_\varepsilon$. The identity \eqref{e5.3}  produces the  following equality in $L^2(\R)$ and so in $H^{-1}(\R)$:
\begin{align}\label{e5.4}
z_\varepsilon (t,\;\cdot\;) 
&= \int^t_0 \d s \left\{ \left[ \( a(s,\;\cdot\;) z (s,\;\cdot\;) \)
\star \phi_\varepsilon \right]'' - \(a(s,\;\cdot\;)z(s,\;\cdot\;)\)\star 
\phi_\varepsilon\right\}\d s  \nonumber\\
&+ \int^t_0 \d s(a(s,\;\cdot\;)z(s,\;\cdot\;))\star \phi_\varepsilon\\
&+ \sum_{i=1}^N \int^t_0 \d W^i_s  (e_i z)(s,\;\cdot\;) \star \phi_\varepsilon. \nonumber
\end{align}
We apply $(I-\Delta)^{-1}$ and we get
\begin{align} \label{e5.4bis}
(I-\Delta)^{-1} z_\varepsilon (t,\;\cdot\;) &= -\int^t_0 \d s (a(s,\;\cdot\;)
 z(s,\;\cdot\;)) \star \phi_\varepsilon\\
&+ \int^t_0 \d s (I-\Delta)^{-1} \left[ \(a(s,\;\cdot\;) z (s,\;\cdot\;) \)\star \phi_\varepsilon \right]\\
&+ \sum_{i=1}^N \int^t_0 \d W^i_s (I-\Delta)^{-1} (e^i z)(s,\;\cdot\;) \star \phi_\varepsilon.
\end{align}
We apply It\^o's formula for stochastic calculus, with values 
in the Hilbert space $H^{-1}(\R)$. For a general introduction to
 infinite dimensional Hilbert valued calculus, see \cite{dpz} or \cite{rockpre}. We evaluate the $H^{-1}$-norm of $z_\varepsilon(t)$. 
Taking into account, \eqref{e5.4}, \eqref{e5.4bis} and that $\left<f,g\right>_{H^{-1}}=\left<f,(I-\Delta)^{-1}g\right>_{L^2}$, it gives
\begin{align}\label{e5.5}
 g_\varepsilon (t) &= 2\int^t_0 \left<z_\varepsilon (s,\;\cdot\;),\;\d z_\varepsilon (s,\;\cdot\;)\right>_{H^{-1}}  \\
&+ \sum_{i=1}^N \int^t_0 \d s\left< (z (s,\;\cdot\;) e^i) \star \phi_\varepsilon, \;  
(z (s,\;\cdot\;) e^i) 
 \star \phi_\varepsilon  \right>_{H^{-1}} \\
&= -2\int^t_0 \< z_\varepsilon (s,\;\cdot\;),\(a(s,\;\cdot\;)z(s,\;\cdot\;)\) \star \phi_\varepsilon \>_{L^2} \d s  \\
&+ 2\int^t_0 \d s \< z_\varepsilon (s,\;\cdot\;),\; (I-\Delta)^{-1} \( \( a(s,\;\cdot\;)z(t,\;\cdot\;)\)\star\phi_\varepsilon \)\>_{L^2} \\
&+ \sum_{i=1}^N \int^t_0 \d s\left< (z (s,\;\cdot\;) e^i) \star \phi_\varepsilon,
 \;  (z (s,\;\cdot\;) e^i) 
 \star \phi_\varepsilon  \right>_{H^{-1}} \\
&+ 2  \int^t_0 \d s \< z_\varepsilon (s,\;\cdot\;), (ze^0)(s,\cdot)
 \star \phi_\varepsilon \right>_{H^{-1}} + M^\varepsilon_t
\end{align}
 where
\begin{equation}\label{e5.6}
  M^\varepsilon_t = 2 \sum^N_{i=1} \int_0^t 
\left< z_\varepsilon (s,\;\cdot\;),\; (e^i z)(s,\;\cdot\;) \star \phi_\varepsilon
\right>_{H^{-1}} \d W^i_s.
\end{equation}
Below we will justify that \eqref{e5.6} is well-defined.
We summarize \eqref{e5.5} into
\begin{equation*}  
g_\varepsilon(t) = \wt g_\varepsilon(t)+M^\varepsilon_t, t \in [0,T].
\end{equation*}
We remark that 
\begin{align} \label{e5.6bis}
\sum^N_{i=1} \int_0^T \left (\< z
 (s,\;\cdot\;),\; e^i z (s,\;\cdot\;) \>_{H^{-1}}\right)^2  \d s 
&= \sum^N_{i=1} \int_0^T \left (\< e^i z (s,\;\cdot\;),\; (I - \Delta)^{-1} z (s,\;\cdot\;) \>_{L_2}\right)^2  \d s  \\
& \le \sum^N_{i=1} \int_0^T \Vert  e^i z (s,\;\cdot\;)\Vert_{L^2}^2 \Vert z(s,\cdot) \Vert_{H^{-2}}^2 \d s \\
&\le  \sum_{i=1}^N \Vert e^i\Vert_\infty^2
\sup_{s \in [0,T]} \Vert z(s,\cdot) \Vert_{H^{-2}}^2 
 \int_0^T \Vert z (s,\;\cdot\;)\Vert_{L^2}^2  \d s,  
\end{align}
because $z: [0,T] \rightarrow H^{-2}$ is a.s. continuous by Remark \ref{R5.2} b).

Consequently
\begin{equation*}
M_t = \sum^N_{i=1} \int^t_0 \< z (s,\;\cdot\;),\; e^i z(s,\;\cdot\;)\>_{H^{-1}}
 \d W^i_s
\end{equation*}
is a well-defined local martingale.
It is also not difficult to show that 
for $\varepsilon > 0$, 
$$ \int^T_0 \left\{ \sum^N_{i=1} \< z_\varepsilon
  (s,\;\cdot\;),\; (e^i z(s,\;\cdot\;)) \star \phi_\varepsilon \>_{H^{-1}}^2 \right\}^2 \d s < \infty,$$
and so $M^\varepsilon$ is a local martingale.

By assumption we have of course $(\o \not \in N_0)$
\begin{equation}\label{e5.7}
\int_{[0,T]\times \R} (z_\varepsilon (s,\xi)-z(s,\xi))^2 \d s \d \xi \substack{\longrightarrow\\ \varepsilon \to 0} 0,
\end{equation}
\begin{equation}\label{e5.8}
\int_{[0,T]\times \R} ((az) \star \phi_\varepsilon -az)^2 (s,\xi) \substack{\longrightarrow\\ \varepsilon \to 0} 0,
\end{equation}
\begin{equation}\label{e5.8bis}
\int_{[0,T]\times \R} ((z(s,\cdot) e_i) \star \phi_\varepsilon - z(s,\cdot) e_i)^2 (\xi) \substack{\longrightarrow\\ \varepsilon \to 0} 0,
\end{equation}
for every $ 1 \le i \le N$,
because $z,az, e_i z \in L^2([0,T]\times \R), 1 \le i \le N$.
Using \eqref{e5.7} and \eqref{e5.8bis}, it is not difficult to show that $(\o \notin N_0)$
\begin{equation}\label{e5.9}
\sum^N_{i=0} \int^T_0 \(\< z_\varepsilon (s,\;\cdot\;), (e^i z (s,\;\cdot\;) \star \phi_\varepsilon \> -\< z(s,\;\cdot\;), e^i z(s,\;\cdot\;)\>\)^2 \d s
\end{equation}
converge to zero.
As a consequence of  \eqref{e5.8bis}, 
 $(\o \notin N_0)$, 
\begin{equation}\label{e5.10}
\sum^N_{i=0} \int^T_0 \( \norm{z(s,\;\cdot\;) e^i ) \star \phi_\varepsilon}^2_{H^{-1}}-\norm{z(s,\;\cdot\;)e^i}^2_{H^{-1}}\)\d s
\end{equation}
converges to zero.
Taking into account \eqref{e5.7}, \eqref{e5.8}, \eqref{e5.9} and 
\eqref{e5.10} we obtain $(\o \notin N_0)$, that 
$\lim\limits_{\varepsilon \to 0} \wt g_\varepsilon (t) = \wt g (t),\; t \in [0,T]$,
 where 
\begin{align}\label{e5.11}
\wt g (t) = & -2 \int^t_0 \<z(s,\;\cdot\;), a(s,\;\cdot\;) z(s,\;\cdot\;)\>_{L^2} \d s\\
& + 2 \int^t_0 \d s \<z(s,\;\cdot\;), (I-\Delta)^{-1} (a(s,\;\cdot\;) z(s,\;\cdot\;))\>_{L^2}\\
& + 2 \int^t_0 \d s \<z(s,\;\cdot\;),z(s,\;\cdot\;)e^0\>_{H^{-1}}\\
& + \sum^N_{i=1} \int^t_0 \<z(s,\;\cdot\;)e^i, z(s,\;\cdot\;) e^i\>_{H^{-1}} \d s.
\end{align}
The convergence of the second term in the right-hand side of \eqref{e5.5} 
to the second term of the right hand sides of \eqref{e5.11} works again 
using \eqref{e5.7} and \eqref{e5.8} cutting the difference in two pieces and using Cauchy-Schwarz. On the other hand the convergence of \eqref{e5.9} 
to zero implies that 
$M^\varepsilon \to M$  ucp, so that the ucp limit of $\wt g_\varepsilon (t)
+M^\varepsilon_t$ gives $\wt g(t)+M_t$. So after a possible modification of the
 $P$-null set $N_0$, setting $g(t) := \norm{z(t,\;\cdot\;)}^2_{H^{-1}}$,
 for $\o\notin N_0$, we have
\begin{align}\label{e5.12}
g(t)& +2 \int^t_0 \< z(s,\;\cdot\;), a(s,\;\cdot\;)z(s,\;\cdot\;)\>_{L^2} \d s\\
&= 2 \int^t_0 \d s \< (I-\Delta)^{-1} z(s,\;\cdot\;), a(s,\;\cdot\;)z(s,\;\cdot\;)\>_{L^2}\\
&+ 2 \int^t_0 \d s \< z(s,\;\cdot\;), e^0 z(s,\;\cdot\;)\>_{H^{-1}}\\
&+ \sum^N_{i=1} \int^t_0 \< z(s,\;\cdot\;) e^i, z(s,\;\cdot\;) e^i \>_{H^{-1}} \d s\\
&+ M_t.
\end{align}
By the inequality
\begin{equation*}
2bc \leq \frac{b^2}{\norm{a}_\9} + c^2 \norm{a}_\9,
\end{equation*}
$b,c\in\R$, it follows,
\begin{align*}
2 \int^t_0 & < (I-\Delta)^{-1} z(s,\;\cdot\;),(az)(s,\;\cdot\;) >_{L^2} \d s\\
& \leq \norm{a}_\9 \int^t_0 \norm{(I-\Delta)^{-1} z(s,\;\cdot\;)}^2_{L^2} \d s\\
& + \frac{1}{\norm{a}_\9} \int^t_0 < (az)(s,\;\cdot\;),(az)(s,\;\cdot\;)>_{L^2} \d s\\
& \leq \norm{a}_\9 \int^t_0 \norm{z(s,\;\cdot\;)}^2_{H^{-2}} \d s\\
& + \frac{1}{\norm{a}_\9} \norm{a}_\9 \int^t_0 \< z(s,\;\cdot\;), az(s,\;\cdot\;)\>_{L^2} \d s.\\
\end{align*}
Since $\norm{\; \cdot \;}_{H^{-2}} \leq \norm{\; \cdot \;}_{H^{-1}}$, 
\eqref{e5.12} gives now $(\o \notin N_0)$,
\begin{align*}
g(t) & + \int^t_0 \< z(s,\;\cdot\;),(az)(s,\;\cdot\;) \>_{L^2} \d s\\
\leq M_t & + \sum_{i=1}^N \int^t_0 \d s \< z(s,\;\cdot\;) e^i, z(s,\;\cdot\;) e^i\>_{H^{-1}}\\
& + 2 \int^t_0 \d s \< z(s,\;\cdot\;), z(s,\;\cdot\;) e^0\>_{H^{-1}}\\
& + \norm{a}_\9 \int^t_0 \norm{z(s,\;\cdot\;)}^2_{H^{-1}} \d s.
\end{align*}
Since $e^i,\; 0\leq i\leq N$ are $H^{-1}$-multipliers,
 we obtain the existence of a constant $\shc=\shc (e_i,\;1\leq i \leq n,\; \norm{a}_\9)$ such that
 ($\o \notin N_0$)
\begin{align}\label{e5.13}
g(t) & + \int^t_0 \< z(s,\;\cdot\;),(az)(s,\;\cdot\;) \>_{L^2} \d s\\
& \leq M_t + \shc \int^t_0 \norm{z(s,\;\cdot\;)}^2_{H^{-1}} \d s\\
& = M_t + \shc \int^t_0 g(s) \d s, \; \forall \; t\in [0,T].
\end{align}
We proceed now via localization which is possible because 
$\int^T_0 \norm{z(s,\;\cdot\;)}^2 \d s$ and $\sup_{t \in [0,T]} \Vert z(s,\cdot)\Vert_{H^{-2}}$
are $P$ a.s. finite. 
 Let $(\varsigma^\ell)$ be the sequence of stopping times
\begin{equation*}
\varsigma^\ell:= \inf \{ t\in [0,T] | \int^t_0 \d s \norm{z(s,\;\cdot\;)}^2_{L^2} \ge \ell, \Vert z(s, \cdot) \Vert^2_{H^{-2}} \geq \ell \}.
\end{equation*}
If $\{ \; \}=\emptyset$ we convene that $\varsigma^\ell = +\9$. Clearly the stopped processes $M^{\varsigma^\ell}$ are (square integrable) martingales starting at zero. We evaluate \eqref{e5.13} at $t\wedge\varsigma^\ell$.
Taking the expectation we get
\begin{align*}
E(g(t\wedge\varsigma^\ell)) & \leq \underbrace{E(M_{\varsigma^\ell \wedge t})}_{=0} +
\shc E\(\int_0^{t\wedge\varsigma^\ell} g(s) \d s \)\\
& \leq \shc \int_0^t \d s E(g(s\wedge\varsigma^\ell)).
\end{align*}
By Gronwall lemma it follows
\begin{equation*}
 E(g(t\wedge\varsigma_\ell))=0\quad \forall\; \ell \in \N^\star.
\end{equation*}
Since $g$ is a.s. continuous and $\lim_{\ell \to \infty} t\wedge\varsigma_\ell = T$
 a.s., for every $t\in[0,T]$, by Fatou's lemma we get
$$
E(g(t))  = E\(\liminf_{\ell\to \9} g(t\wedge\varsigma_\ell) \) \leq \liminf_{\ell\to \9} E\( g(t\wedge\varsigma_\ell) \)=0.
$$
Finally the result follows.
\end{proof}

\section{The non-degenerate
 case} \label{S6}

\setcounter{equation}{0}

We are now able to discuss the double probabilistic representation 
of a solution to the \eqref{PME} when $\psi$ is non-degenerate 
provided that its solution fulfills some properties.
 We remark that up to now we have not used the first item of Assumption 
\ref{E1.0}. We remind
 that the functions $ e_i. 0 \le i \le N,$ are
$H^{-1}$-multipliers.

\begin{theo}\label{thm6.1}
We suppose the following assumptions.
\begin{enumerate}
\item $x_0$ is a real probability measure. 
\item $\psi$ is non-degenerate.
\item There is only one  random field
 $X:[0,T]\times \R\times \Omega \to \R$ 
solution of \eqref{PME} (see Definition \ref{DSPDE})
such that 
\begin{equation}\label{e6.1}
\int\limits_{[0,T]\times \R} X^2 (s,\xi) \d s \d \xi < \9 \quad \text{a.s.}
\end{equation}
Then there is a unique weak  solution to the (DSNLD)$(\Phi,\mu,x_0)$.
\end{enumerate}
\end{theo}

\begin{rem} \label{R6.2}  
\begin{enumerate}  
\item 
Suppose that $e^i, 1 \le i \le d$, belong to $W^{1,\infty}$.
 In Theorem 3.4 of \cite{BRR3}, we show that
(even if $x_0$ belongs to $ H^{-1}(\R)$)), when $\psi$  is Lipschitz, there is a solution to \eqref{PME} such that
\begin{equation*}
E\( \int\limits_{[0,T]\times \R} X^2 (s,\xi) \d s \d \xi\) < \9.
\end{equation*}
According to Thereom \ref{TB1}, that solution is unique. \\
In particular item 3. in Theorem \ref{thm6.1} statement holds.
\item
 Theorem \ref{thm6.1} constitutes the converse of Theorem \ref{T32}
 when $\psi$ is non-degenerate.
\item
 Again for simplicity of the formulation, without restriction of 
generality, in the proof we will suppose $\psi$ to be single-valued and $\Phi$ admitting a continuous extension to $\R$. Otherwise one can adopt the techniques of \cite{BRR1}.
\item
 As side-effect of the proof of the weak-strong existence Proposition 
\ref{P4.1}, the space $(\O_0,\shg,Q)$ can be chosen as 
$\O_0=\O_1\times \O,\; \O_1= C([0,T];\; \R) \times \R, \; \shg =\shb(\O_1)\times 
\shf,\; Q(H\times F)=\int\limits_{\O_1\times\O} \d P(\o) 
1_F (\o)Q(\d \o_1,\o)$.
\end{enumerate}
\end{rem}

\begin{proof}
\begin{enumerate}
[1)]
\item We set $\gamma(t,\xi,\o)=\Phi \(X(t,\xi,\o)\right)$. According to
 Proposition \ref{P4.1} there is a weak-strong solution $Y$ to (DSDE)$(\gamma,x_0)$. By Proposition \ref{P4.2} $\o$ a.s. the $\mu$-marginal laws of $Y$ admit  densities $\(q_t (\xi,\o),t\in]0,T],\; \xi \in\R,\; \o \in \O\right)$ such that $\pas$
\begin{equation}\label{e6.2}
\int\limits_{[0,T]\times \R} \d s \d \xi q^2_s (\xi,\;\cdot\;) < \9 \quad \text{a.s.}
\end{equation}
\item Setting 
\begin{equation*}
\nu_t (\xi,\o)= \left(
\begin{array}{ccc}
q_t(\xi,\o) \d \xi &:& t\in ]0,T],\\
x_0 &:& t=0,
\end{array}
\right.
\end{equation*}
$\nu$ is a solution to \eqref{e5.1} with $\nu_0=x_0,\; a(t,\xi,\o)=
\Phi^2(X(t,\xi,\o))$. This can be shown applying It\^o's formula similarly as in the proof of Theorem \ref{T32}.
\item On the other hand $X$ is obviously also a solution of \eqref{e5.1},
which  in particular verifies \eqref{e6.1}.
Consequently $z^1=\nu, \; z^2=X$ verify items i), ii), iii) of
 Theorem \ref{T51}.
So Theorem \ref{T51} implies that $\nu \equiv X$; this shows that $Y$ provides a solution to  (DSNLD)$(\Phi,\mu,x_0)$.
\item Concerning uniqueness, let $Y^1,Y^2$ be two 
solutions to the (DSNLD) related to $(\Phi,\mu,x)$. The 
corresponding random fields $X^1,X^2$ constitute the $\mu$-marginal laws of $Y^1,Y^2$ respectively.
\end{enumerate}
Now $Y^i,\; i=1,2$, is a weak-strong solution of (DSDE)$(\gamma_i,x)$ with
 $\gamma_i(t,\xi,\o)=\Phi(X_i(t,\xi,\o))$, so by Proposition 
\ref{P4.2}
 $X_i,\; i=1,2$ fulfills 
\eqref{e6.1}. By Theorem \ref{T32}, $X_1$ and $X_2$ are solutions to
\eqref{PME}.
 By assumption 3. of the statement,
 $X_1=X_2$. The conclusion follows by Proposition \ref{P4.1}, which  
guarantees the uniqueness of the weak-strong solution of (DSDE)$(\gamma,x_0)$
 with $\gamma_1=\gamma_2$.
\end{proof}

\begin{rem} \label{R6.3}
One side-effect of Theorem \ref{thm6.1} is the following. Suppose $\psi$
 to be non-degenerate. Let $X:[0,T]\times \R\times\O\to\R$ be a solution such that $\pas$
\begin{equation*}
\int\limits_{[0,T]\times \R} X^2 (s,\xi) \d s \d \xi <\9 \quad \text{a.s.}
\end{equation*}
We have the following for $\o \pas$.
\begin{enumerate}[i)]
\item $X(t,\; \cdot\;,\o)\geq 0$ a.e. $\forall\; t\in [0,T]$
\item $E\(\int\limits_\R X(t,\xi)\d \xi\right)=1,\; \forall \; t\in [0,T]$ if $e_0=0$.
\end{enumerate}
\end{rem}
\begin{rem}\label{R6.4}
If \eqref{PME} has a solution, not necessarily unique then (DSNLD) 
with respect to $(\Phi,\mu,x_0)$ still admits existence.
\end{rem}

\section{The degenerate case}
\label{S7}

\setcounter{equation}{0}

The idea consists in proceeding similarly to \cite{BRR2}, which treated
 the case $\mu=0$ and the case when $x_0$ is absolutely continuous
with bounded density.
$\psi$ will be assumed to be strictly increasing after some zero $u_c \ge 0$,
see Definition \ref{DNond}.
We recall that if $\psi$ is degenerate, then necessarily
$\Phi(0):=  \lim_{\varepsilon \rightarrow 0} \Phi(x) = 0$.
\begin{rem} \label{R7.2}
\begin{enumerate}[i)]
\item If $u_c>0$ then $\psi$ is necessarily degenerate
and also $\Phi$ restricted to $[0,u_c]$ vanishes. 
\item Let 
 $x_0$ is a probability on $\R$. Suppose the existence of a 
solution $X$ to \eqref{PME} such that
\begin{equation}\label{e7.1}
E\(\int_{[0,T]\times\R} \d t\d \xi X^2 (t,\xi)\)< \9.
\end{equation}
We recall that, by Definition \ref{DSPDE}, a.s. it belongs to $C\([0,T],\shs'(\R)\)$.
 In this case a.s. $\int\limits^t_0 \psi (X(s,\;\cdot\;))\d s\in H^1 (\R)$
 for every $t\in [0,T]$. See Remark \ref{remB2} vi). 
\item If $X$ is a solution such that \eqref{e7.1} is verified and $x_0 \in H^{-2}$,
then $X \in C([0,T]; H^{-2})$ a.s., see Remark \ref{remB2} vi).
\item \eqref{e7.1} implies in particular that if $X$ is a solution of \eqref{PME}, then
\begin{equation}\label{e7.2}
E\(\int_0^T \d s \norm{X(s,\;\cdot\;)}^2_{H^{-1}}\)<\9.
\end{equation}
\item If $\psi$ is Lipschitz, we remind (Remark \ref{R6.2} 1.) and
 $x_0 \in L^2$,
 there is a unique solution to \eqref{PME} 
such that \eqref{e7.1} is fulfilled, at least if we suppose that
all the $e^i$ belong to $H^1(\R)$, see Theorem 3.4 of
 \cite{BRR3} and Theorem \ref{TB1}.
\end{enumerate}
\end{rem}
\begin{theo} \label{T73}
We suppose the following.
\begin{enumerate}
\item The functions $ e_i. 1 \le i \le N$ 
 belong to $H^1(\R)$. 
\item We suppose that $\psi:\R\to\R$ is non-decreasing, Lipschitz and 
strictly increasing after some zero.
\item Let $x_0$ belong to $L^2(\R)$.
\end{enumerate}
 Then there is a weak solution to the (DSNLD)$(\Phi, \mu, x_0)$.
\end{theo}
\begin{proof}[Proof]
\begin{enumerate}[1)]
\item We proceed by approximation rendering $\Phi$ non-degenerate. Let $\kappa>0$. We define $\Phi_\kappa:\R\to\R_+$ by
\begin{align*}
\Phi_\kappa (u) = \sqrt{\Phi^2(u)+\kappa}\\
\psi_\kappa (u) = \Phi_\kappa^2(u)\cdot u
\end{align*}
Let $X^\kappa$ be the solution so \eqref{PME} with $\psi_\kappa$ instead of $\psi$.
According to Theorem \ref{thm6.1} and Remark \ref{R6.2} 4.,
 setting $\O_1=C\([0,T],\R\)\times \R, \;\shh$ its Borel $\sigma$-algebra, $Y(\o^1,\o)=\o^1$, there are families of probability kernels $Q^\kappa$ on $\shh\times\O_1$, 
and processes $B^\kappa$ on $\O_0$ such that
\begin{enumerate}[i)]
\item $B^\kappa (\;\cdot\;,\o)$ is a
 $Q^\kappa (\;\cdot\;,\o)$-Brownian motion;
\item 
\begin{equation}\label{e7.3}
Y_t=Y_0 +\int\limits_0^t \Phi_\kappa (X^\kappa(s,Y_s,\o))\d B_s^\kappa; 
\end{equation}
\item $Y_0^\kappa$ is distributed according to $x_0=X^\kappa(0,\;\cdot\;)$.
\item The $\mu$-marginal laws of $Y$ under $Q^\kappa$ are $(X^\kappa(t,\;\cdot\;))$.
\end{enumerate}
We need to show the existence of a probability kernel $Q$ on
 $\shh\times\O_1$, a process $B$ on $\O_0$ such that 
the following holds.
\begin{enumerate}[i)]
\item $B (\;\cdot\;,\o)$ is a $Q (\;\cdot\;,\o)$-Brownian motion.
\item 
\begin{equation}\label{e7.4}
Y_t=Y_0+\int\limits_0^t \Phi (X(s,Y_s,\o))\d B_s^\kappa.
\end{equation}
\item $Y_0$ is distributed according to $x_0$.
\item For every $t\in]0,T],\; \varphi \in C_b(\R)$,
\begin{equation}\label{e7.5}
\int_\R X(t,\xi) \varphi(\xi)\d \xi = E^{Q^\o} \left(\varphi(Y_t)
\she_t\(\int\limits_0^\cdot \mu (\d s, Y_s) X(s,Y_s)\)\right) ,
\end{equation}
where $Q^\o = Q(\;\cdot\;,\o)$.
\end{enumerate}
\item We need to show that $X^\kappa$ approaches $X$ in some sense
 when $\kappa \to 0$, where $X$ is
 the solution to \eqref{PME}. This is given in the Lemma \ref{L7.4} below.
\end{enumerate}

\begin{lemma} \label{L7.4}
Under the the assumptions of Theorem \ref{T73}, let $X$ (resp. $X^\kappa$) be a solution of \eqref{PME}
verifying \eqref{e7.1} with $\psi(u) = u \Phi^2(u)$ (resp. $\psi_\kappa(u)
 = u(\Phi^2(u) + \kappa)$), for $ u > 0$.
We have the following.
\begin{enumerate}[a)]
\item
$
\lim_{\kappa \to 0} \sup_{t\in [0,T]} E\(\norm{X^\kappa (t,\;\cdot\;)-X(t,\;\cdot\;)}^2_{H^{-1}}\)=0;
$
\item
$\lim_{\kappa \to 0}E\(\int^T_0 \d t \norm{\psi\(X^\kappa (t,\;\cdot\;)\)-\psi\(X(t,\;\cdot\;)\)}^2_{L^2}\)=0;$
\item  $\lim_{\kappa \to 0} \kappa  E\(\int_{[0,T]\times \R}
 \d t \d \xi \(X^\kappa (t,\xi) - X(t,\xi)\)^2 \) = 0. $
\end{enumerate}
\end{lemma}
\begin{rem} \label{R7.5}
\begin{enumerate}[1)]
\item a) implies of course
\begin{equation*}
\lim_{\kappa \to 0}E\(\int^T_0 \d t \norm{X^\kappa (t,\;\cdot\;)-X(t,\;\cdot\;)}^2_{H^{-1}}\)=0.
\end{equation*}
\item In particular Lemma \ref{L7.4} b) implies that for each sequence $(\kappa_n)\to 0$ there is a subsequence still denoted by the same notation that
\begin{equation*}
\int\limits_{[0,T]\times\R} \(\psi (X^{\kappa_n} (t,\xi))-\psi(X(t,\xi))\)^2 \d t \d \xi \substack{\longrightarrow\\n\to\9}0 
\end{equation*}
a.s.
\item 
For every $t\in[0,T]$
\begin{equation*}
X(t,\;\cdot\;) \geq 0 \quad  d\xi \otimes dP \text{a.e.}
\end{equation*}
Indeed, for this it will be enough to show that a.s.
\begin{equation}\label{e7.6}
\int\limits_\R \d \xi \varphi(\xi)X(t,\xi) \ge 0 \text{ for every }
 \varphi \in \shs(\R),
\end{equation}
for every $t \in [0,T]$.
Since $X\in C\([0,T];\;\shs'(\R)\)$ it will be enough to show \eqref{e7.6} for almost all $t\in[0,T]$. This follows since item 1) in this Remark
\ref{R7.5}, implies the existence of a sequence $(\kappa_n)$ such that
\begin{equation*}
\int\limits^T_0 \d t\norm{X^{\kappa_n}(t,\cdot) - X(t,\;\cdot\;)}^2_{H^{-1}}
\substack{\longrightarrow\\n\to\9}0, \quad {\rm a.s.} 
\end{equation*}
\item 
Since $\psi$ is strictly increasing after $u_c$, for $P$ almost all $\o$, for almost all $(t,\xi)\in [0,T]\times\R$,
 there is a sequence $(\kappa_n)$ such that
\begin{equation}\label{e7.7}
\(X^{\kappa_n}(t,\xi)-X(t,\xi)\) 1_{\{X(t,\xi)>u_c\}}(t,\xi)
\substack{\longrightarrow\\n\to\9}0.
\end{equation}
This follows from item 2) of Remark \ref{R7.5}.\\
Since $\Phi^2(u)=0$ for $0\leq u\leq u_c$ and $X$ is a.e. non-negative, 
this implies that $dt d\xi dP$ a.e. we have
\begin{equation}\label{e7.8}
\Phi^2\(X(t,\xi)\)\(X^{\kappa_n}(t,\xi)-X(t,\xi)\)\substack{\longrightarrow\\n\to\9}0.
\end{equation}
\end{enumerate}
\end{rem}
\begin{proof}[Proof (of Lemma \ref{L7.4})]
Proceeding similarly as
 in Theorem \ref{T51}, we can write $\pas$ the following  $H^{-2}(\R)$-valued
 equality.
\begin{align*}
\(X^\kappa -X\)(t,\;\cdot\;)= &\int\limits^t_0 \d s\( \psi_\kappa \(X^\kappa(s,\;\cdot\;)\)-\psi \(X(s,\;\cdot\;)\)\)''\\
& + \sum^N_{i=0} \int\limits^t_0  
\(X^\kappa(s,\;\cdot\;)-X(s,\;\cdot\;)\) e^i \d W_s^i.
\end{align*}
So
\begin{align*}
(I-\Delta)^{-1} \(X^\kappa -X\)(t,\;\cdot\;)= & -\int\limits^t_0 \d s\( \psi_\kappa \(X^\kappa(s,\;\cdot\;)\)-\psi \(X(s,\;\cdot\;)\)\)\\
& + \int\limits^t_0 \d s (I-\Delta)^{-1} \( \psi_\kappa \(X^\kappa(s,\;\cdot\;)\)-\psi \(X(s,\;\cdot\;)\)\)\\
& + \sum^N_{i=0} \int\limits^t_0 (I-\Delta)^{-1} \( e^i \(X^\kappa(s,\;\cdot\;)-X(s,\;\cdot\;)\)\)\d W_s^i. 
\end{align*}
After regularization and application of It\^o calculus with values in $H^{-1}$, setting $g^\kappa(t)=\norm{\(X^\kappa - X\)(t,\;\cdot\;)}^2_{H^{-1}}$,
similarly to the proof of Theorem \ref{T51}, we obtain
\begin{align}\label{e7.9}
g^\kappa(t) = & \sum^N_{i=1} \int\limits^t_0 \norm{e^i \(X^\kappa -X\)(s,\;\cdot\;)}^2_{H^{-1}} \d s\\
- & 2 \int\limits^t_0 \<\(X^\kappa -X\)(s,\;\cdot\;), \psi_\kappa \(X^\kappa(s,\;\cdot\;)\)-\psi \(X(s,\;\cdot\;)\)\>_{L^2}\\
+ & 2 \int\limits^t_0 \d s \< \(X^\kappa -X\)(s,\;\cdot\;),(I-\Delta)^{-1}
\left(\psi_\kappa \(X^\kappa(s,\;\cdot\;)\)-\psi \(X(s,\;\cdot\;)\) \right)
\>_{L^2}\\
+ & 2 \int\limits^t_0 \d s \< \(X^\kappa -X\)(s,\;\cdot\;),(I-\Delta)^{-1} e^0 \(X^\kappa -X\)(s,\;\cdot\;)\>_{L^2}\\
& + M^\kappa_t,
\end{align}
where $M^\kappa$ is the local martingale
\begin{equation*}
M^\kappa_t = 2 \sum^N_{i=1} \int\limits^t_0 \< (I-\Delta)^{-1} \(X^\kappa -X\)(s,\;\cdot\;),\(X^\kappa -X\)(s,\;\cdot\;) e^i\>_{L^2} \d W^i_s.
\end{equation*}
Indeed, $M^\kappa$ is a local martingale because, taking into account
\eqref{eFB1} and Remark \ref{R7.2} iii), acting similarly as for the proof 
of \eqref{e5.6bis}, see also \eqref{EFB100} in  Appendix B,
we can prove that
$$  \sum^N_{i=1} \int\limits^t_0 \vert \< \(  X^\kappa  -  X \vert\)
(s,\;\cdot\;),(I-\Delta)^{-1}\(X^\kappa -X\)(s,\;\cdot\;)
 e^i\>_{L^2} \vert^2 \d s < \infty.$$
\eqref{e7.9} gives
\begin{align}
g^\kappa(t) & + 2\int\limits^t_0 \<\(X^\kappa-X\)(s,\;\cdot\;),\;\psi \(X^\kappa(s,\;\cdot\;)\)-\psi \(X(s,\;\cdot\;)\)\>_{L^2} \d s\\
& + 2\kappa \int\limits^t_0 \<\(X^\kappa-X\)(s,\;\cdot\;),\;\(X^\kappa-X\)(s,\;\cdot\;)\>_{L^2} \d s\\
\leq & -2\kappa \int\limits^t_0 \d s\<\(X^\kappa-X\)(s,\;\cdot\;),\; X(s,\;\cdot\;)\>_{L^2} \d s\\
& + \sum^N_{i=1}\int\limits^t_0 \norm{e^i \(X^\kappa -X\)(s,\;\cdot\;)}^2_{H^{-1}} \d s\\
& + 2 \int\limits^t_0 \d s\<(I-\Delta)^{-1}\(X^\kappa -X\)(s,\;\cdot\;),\;\psi \(X^\kappa(s,\;\cdot\;)\)-\psi \(X(s,\;\cdot\;)\)\>_{L^2} 
\end{align}
\begin{align}
& + 2\kappa \int\limits^t_0 \d s\<(I-\Delta)^{-1}\(X^\kappa -X\)(s,\;\cdot\;),\;\(X^\kappa-X\)(s,\;\cdot\;)\>_{L^2}\\
& + 2\kappa \int\limits^t_0 \d s\<(I-\Delta)^{-1}\(X^\kappa -X\)(s,\;\cdot\;),\; X(s,\;\cdot\;)\>_{L^2}\\
& + 2 \int\limits^t_0 \d s\< (I-\Delta)^{-1} \(X^\kappa -X\)(s,\;\cdot\;),\;
\(e^0\(X^\kappa -X\)(s,\;\cdot\;)\)\>_{L^2}  + M^\kappa_t .
\end{align}
We use Cauchy-Schwarz and the inequality
\begin{equation*}
2\sqrt{\kappa} b \sqrt{\kappa}c\leq \kappa b^2 + \kappa c^2
\end{equation*}
with first
$$
b  = \norm{X^\kappa(s,\;\cdot\;)-X(s,\;\cdot\;)}_{L^2}, \quad
c  = \norm{X(s,\;\cdot\;)}_{L^2}
$$
and then
$$
b  = \norm{X^\kappa(s, \cdot) -X (s,\cdot)}_{H^{-2}}, \quad
c  = \norm{X(s,\;\cdot\;)}_{L^2}.
$$
We also take into account the property of $H^{-1}$-multiplier 
for $e^i,\; 0\leq i\leq N$. Consequently there is a constant
 ${\mathrm C}(e)$ depending on $(e^i,\; 0\leq i\leq N)$ such that
\begin{align}\label{e7.11}
g^\kappa(t) & + 2\int\limits^t_0 \<\(X^\kappa-X\)(s,\;\cdot\;),\;\psi \(X^\kappa(s,\;\cdot\;)\)-\psi \(X(s,\;\cdot\;)\)\>_{L^2} \d s\\
& + 2\kappa \int\limits^t_0 \norm{X^\kappa(s,\;\cdot\;)-X(s,\;\cdot\;)}^2_{L^2} \d s\\
\leq & \kappa \int\limits^t_0 \norm{\(X^\kappa -X\)(s,\;\cdot\;)}^2_{L^2} \d s\\
& + \kappa \int\limits^t_0 \d s \norm{X(s,\;\cdot\;)}_{L^2}^2\\
& + {\mathrm C}(e)\int\limits^t_0 \d s \norm{X^\kappa(s,\;\cdot\;)-X(s,\;\cdot\;)}^2_{H^{-1}}\\ 
& + 2\int\limits^t_0 \norm{\(X^\kappa -X\)(s,\;\cdot\;)}_{H^{-2}} \norm{\psi \(X^\kappa(s,\;\cdot\;)\)-\psi \(X(s,\;\cdot\;)\)}_{L^2}\\
& + 2\kappa \int\limits^t_0 \d s g^\kappa(s)\\
& + \kappa \int\limits^t_0 \d s \norm{\(X^\kappa -X\) (s,\;\cdot\;)}_{H^{-2}}^2 
+ \kappa \int\limits^t_0 \d s \norm{X(s,\;\cdot\;)}_{L^2}^2\\
& + M^\kappa_t .
\end{align}
Since $\psi$ is Lipschitz, it follows
\begin{equation*}
\(\psi(r)-\psi(r_1)\)(r-r_1) \geq \alpha\(\psi(r)-\psi(r_1)\)^2,
\end{equation*}
for some $\alpha>0$. Consequently, the inequality
\begin{equation*}
2bc \leq b^2 \alpha+\frac{c^2}{\alpha},
\end{equation*}
with $b,c\in\R$ and the fact that $\norm{\;\cdot\;}_{H^{-2}}\leq \norm{\;\cdot\;}_{H^{-1}}$ give
\begin{align}\label{e7.12}
2 & \int\limits^t_0 \d s \norm{\(X^\kappa -X\)(s,\;\cdot\;)}_{H^{-2}} \norm{\psi \(X^\kappa(s,\;\cdot\;)\)-\psi \(X(s,\;\cdot\;)\)}_{L^2}\\
\leq & \int\limits^t_0 \d s \alpha g^\kappa (s,\;\cdot\;)+ \int\limits^t_0
 \d s\<\psi \(X^\kappa(s,\;\cdot\;)\)-\psi \(X(s,\;\cdot\;)\),\;X^\kappa(s,\;\cdot\;)-X(s,\;\cdot\;)\>_{L^2}.
\end{align}
So \eqref{e7.11} yields 
\begin{align}\label{e7.13}
g^\kappa(t) & + \int\limits^t_0 \< X^\kappa(s,\;\cdot\;)-X(s,\;\cdot\;),\;\psi \(X^\kappa(s,\;\cdot\;)\)-\psi \(X(s,\;\cdot\;)\)\>_{L^2} \d s\\
& + \kappa \int\limits^t_0 \d s \norm{X^\kappa(s,\;\cdot\;)-X(s,\;\cdot\;)}^2_{L^2} \d s\\
\leq & 2\kappa \int\limits^t_0 \d s \norm{X(s,\;\cdot\;)}^2_{L^2}\\
& + M^\kappa_t\\
& + \({\mathrm C}(e) +\alpha+3\kappa\)\int\limits^t_0 g^\kappa (s)\d s.
\end{align}
Taking the expectation we get
\begin{align*}
 E(g^\kappa(t)) &\le  \({\mathrm C}(e) + \alpha + 3\kappa\) \int\limits^t_0 E(g^\kappa (s)) \d s \\
&+ 2 \kappa \int_0^t E(\Vert X(s,\cdot) \Vert^2_{L^2}) \d s,
\end{align*}
for every $t\in[0,T]$.
By Gronwall lemma we get
\begin{equation} \label{e7.14bis}
E\(g_\kappa(t)\) \leq 2 \kappa E \left\{ \int\limits^T_0 \d s
\norm{X(s,\;\cdot\;)}^2_{L^2}\right\}e^{\({\mathrm C}(e) +\alpha+3\kappa\)T}, 
\quad \forall \; t\in [0,T].
\end{equation}
Taking the supremum and letting $\kappa \to 0$, item a) of Lemma \ref{L7.4}
 is now established. \\
We go on with item b). Since $\psi$ is Lipschitz, \eqref{e7.13} implies that,
 for $t\in [0,T]$,
\begin{align*}
& \int\limits^t_0 \d s\norm{\psi \(X^\kappa(s,\;\cdot\;)\)-\psi \(X(s,\;\cdot\;)\)}^2_{L^2}\\
\leq & \frac{1}{\alpha} \d s \<\psi \(X^\kappa(s,\;\cdot\;)\)-\psi \(X(s,\;\cdot\;)\),\; X^{(\kappa)}(s,\;\cdot\;)-X(s,\;\cdot\;)\>_{L^2}\\
\leq & \frac{\kappa}{2\alpha} \int\limits^t_0 \d s \norm{X(s,\;\cdot\;)}^2_{L^2}\\
& + {\mathrm C}(e,\alpha)  \int\limits^t_0 g_\kappa(s) \d s + M^\kappa_t,
\end{align*}
where ${\mathrm C}(e,\alpha)$ is a constant depending on
 $e^i, 0\leq i\leq N$ and $\alpha$.
Taking the expectation for $t=T$, we get
\begin{align*}
& E\(\int\limits^T_0 \d s\norm{\psi \(X^\kappa(s,\;\cdot\;)\)-\psi \(X(s,\;\cdot\;)\)}_{L^2}^2\)\\
\leq & \frac{\kappa}{2 \alpha}  E\(\int\limits^T_0 \d s\norm{X(s,\;\cdot\;)}^2_{L^2}\)+ {\mathrm C}(e,\alpha) \int\limits^T_0 E(g_\kappa(s))\d s.
\end{align*}
Taking $\kappa \to 0$, 
 \eqref{e7.1}  and \eqref{e7.14bis} provide the conclusion of
 item b) of Lemma \ref{L7.4}.
\begin{enumerate}[c)]
\item Coming back to \eqref{e7.13}, and $t=T$, we have 
\begin{align*}
& 
\kappa \int\limits^T_0 \d s \norm{X^\kappa(s,\;\cdot\;)-X(s,\;\cdot\;)}_{L^2}^2\\
\leq & 2 \kappa \int\limits^T_0 \d s \norm{X(s,\;\cdot\;)}_{L^2}^2 + M^\kappa_T\\
& + \({\mathrm C}(e) +\alpha+3\kappa\) \int\limits^T_0 \d s g^\kappa(s).
\end{align*}
Taking the expectation we have
\begin{align*}
& \kappa E\(\int\limits^T_0 \d s\norm{X^\kappa(s,\;\cdot\;)-X(s,\;\cdot\;)}_{L^2}^2\)\\
\leq & 2\kappa E \left( \int\limits^T_0 \d s\norm{X(s,\;\cdot\;)}^2_{L^2}\right)\\
& + \({\mathrm C}(e) +\alpha+3\kappa\) E\(\int\limits^T_0 g^\kappa(s) \d s\).
\end{align*}
Using item a) and the fact that
\begin{equation*}
E\(\int\limits_{[0,T]\times\R} X^2 (s,\xi)\d s \d \xi\)<\9, 
\end{equation*}
the result follows.
Lemma \ref{L7.4} is finally completely established.
\end{enumerate}
\end{proof}
We need now another intermediate lemma concerning the paths of a solution to \eqref{PME}.
\begin{lemma} \label{L7.6}
For almost all $\o\in\O$, almost all $t\in[0,T]$,
\begin{enumerate}[1)]
\item $\xi \mapsto \psi(X(t,\xi,\o))\in H^1(\R)$,
\item $\xi \mapsto \Phi(X(t,\xi,\o))$ is continuous.
\end{enumerate}
\end{lemma}
\begin{proof}[Proof]
Item 1) is established in \cite{BRR3}, see Definition 
3.2 and Theorem 3.4. 1) implies that $\xi 
\mapsto \psi(X(t,\xi,\o))$ is continuous. By the same arguments as in Proposition 4.22 in \cite{BRR2}, we can deduce item 2).
\end{proof}
\begin{description}
\item{3)} We go on with the proof of Theorem \ref{T73}. 
We keep in mind i), ii), iii), iv) near \eqref{e7.3}. Since $\Phi$ is bounded, using Burkholder-Davies-Gundy inequality are obtains
\begin{equation*}
E^{Q^{\kappa}(\;\cdot\;,\o)}\(Y_t-Y_s\)^4 \leq \const(t-s)^2.
\end{equation*}
On the other hand, for all 
$Q^\kappa (\;\cdot\;,\o), \; Y_0$ is distributed according to $x_0$. By Kolmogorov-Centsov criterion, see for instance an easy adaptation of \cite{KARSH}, Problem 4.11 of Section 2.4, for $\o\in\O$ a.s., the probabilities $Q^\kappa (\;\cdot\;,\o), \; \kappa>0$ are tight. Consequently there is a sequence $Q^{\kappa_n} (\;\cdot\;,\o),\;(\kappa_n)$ depending on $\o$, converging weakly to some probability $Q (\;\cdot\;,\o)$ on $C\([0,T];\;\R\)$.
By Skorohod's theorem there is a new probability space $\(\O_1^\o,\shh^\o,
\tilde Q^\o\)$ and processes $Y^n(\;\cdot\;,\o)$ distributed as $Y(\;\cdot\;,\o)$ under $Q^{\kappa_n}(\;\cdot\;,\o)$, converging to some process
 $Y^\9(\;\cdot\;,\o)$ ucp under $\tilde Q^\omega$.
From now on we will denote again $ Q^\omega := \tilde Q^\omega $.
 In particular $Y^n(\;\cdot\;,\o)$ are local martingales with respect to their own filtrations such that $[Y^n(\;\cdot\;,\o)]_t = \int^t_0 \Phi^2_{\kappa_n}
\(X^{\kappa_n}(r,Y^n_r,\o)\)\d r$. We remark that $Y^n(\;\cdot\;,\o)-Y_0^n
(\cdot,\o)$
 are even square integrable martingales.
\item{4)} Let $X$ be the solution to the SPDE  \eqref{PME}. The next step
 consists in showing that for $P$ almost $\o,\; Y^\9(\;\cdot\;,\o)$ is a weak solution to the equation
\begin{equation}\label{e7.15}
Y^\9_t (\;\cdot\;,\o) = Y_0^\9 (\;\cdot\;,\o)+ \int^t_0 \Phi\(X(s,Y^\9_s, \o)\)\d \beta_s^\o,
\end{equation}
for some Brownian motion $\beta^\o$. We need here a technical lemma.
\end{description}
\begin{lemma} \label{L7.7}
For $\o \; \pas$, the random variables
\begin{equation}\label{e7.16}
\int^T_0 \(\Phi_{\kappa_n}\(X^{\kappa_n}(r,Y^n_r,\o)\)-
\Phi\(X (r,Y^\infty_r,\o)\)\)^2\d r
\end{equation}
converge to zero in $L^p(\O_1^\o, Q^\o),\; \forall\; p\geq 1$, 
and consequently in probability.
\end{lemma}
\begin{proof}[Proof]
It is of course enough to show that the $E^{Q^\o}$ 
expectation of \eqref{e7.16} goes to zero up to a subsequence.
This is bounded by $I_1(n)+I_2(n)$ where
\begin{align*}
I_1(n) = & 2E^{Q^\o}\(\int\limits^T_0 \d r\(\Phi_{\kappa_n} \( X^{\kappa_n}(r,Y^n_r,\o)\)-\Phi \(X(r,Y^n_r,\o)\)^2\)\)\\
I_2(n) = & 2E^{Q^\o}\(\int\limits^T_0 \d r\(\Phi \( 
X(r,Y^n,\o)\)-\Phi \(X(r,Y^\infty,\o)\)^2\)\).
\end{align*}
Since $\Phi\(X(r,\;\cdot\;)\)$ is continuous for almost all 
$(r,\o_1)\in[0,T]\times\O_1^\o$, by Lemma \ref{L7.6}, 
then $I_2(n)\substack{\longrightarrow\\n\to\9}0$ by an
 easy application of Lebesgue's dominated convergence theorem. Concerning $I_1(n)$, it is also enough to show the existence of a subsequence $(\kappa_{n_\ell})$ such that
\begin{equation*}
\(\Phi_{\kappa_n}\(X^{\kappa_n}(r,Y^n_r,\o)\)-\Phi\(X(r,Y^n_r,\o)\)\)^2
\substack{\longrightarrow\\n\to \9}0,
\end{equation*}
$\d Q^\o\times \d r$ a.e.. Since the Dol\'eans
 exponential is strictly positive, this will be guaranteed if we show that
\begin{equation*}
\(\Phi_{\kappa_n}\(X^{\kappa_n}(r,Y^n_r,\o)\)-\Phi\(X(r,Y^n_r,\o)\)\)^2 \she_r
 \(\int_0^\cdot \mu(\d s,Y_s^n)\)\substack{\longrightarrow\\n\to \9}0,
\end{equation*}
$\d Q^\o\times \d r$ a.e. Clearly, this will be verified, if we show that, for any $\varphi:[0,T]\times\R\to\R_+ $ continuous, with compact support, we have
\begin{equation}\label{e7.17}
E^{Q^\o}\(\int\limits^T_0 \varphi(r,Y^n_r)\she_r\(\int\limits^\cdot_0
 \mu(\d s, Y_s^n)\)
\(\Phi_{\kappa_n}\(X^{\kappa_n}(r,Y^n_r,\o)\)-\Phi\(X(r,Y^n_r,\o)\)\)^2 dr
\) 
\end{equation}
goes to zero, eventually up to a subsequence. \\
Since $X^{\kappa_n}$ constitute the $\mu$-marginal laws of $Y^n$, previous expression gives
\begin{align*}
& \int\limits^T_0 \d r \int\limits_\R \varphi(r,y)\(\Phi_{\kappa_n}
\(X^{\kappa_n}(r,y,\o)\)-\Phi\(X(r,y,\o)\)\)^2 X^{\kappa_n}(r,y,\o) \d y\\
\leq & I_{11}(n)+I_{12}(n)+I_{13}(n)+I_{14}(n)
\end{align*}
where
\begin{align*}
I_{11}(n) = & \int\limits^T_0 \d r \int\limits_\R \d y \vert \varphi(r,y) \vert 
\abs{\psi\(X^{\kappa_n}(r,y,\o)\)-\psi\(X(r,y,\o)\)},\\
I_{12}(n) = & \int\limits^T_0 \d r \int\limits_\R \d y 
\vert \varphi(r,y) \vert 
\Phi^2 \(X^{\kappa_n}(r,y,\o)\) \left \vert  (X-X^{\kappa_n})(r,y,\o)
 \right \vert,\\
I_{13}(n) = & \int\limits^T_0 \d r \int\limits_\R \d y
\vert  \varphi(r,y) \vert \kappa_n \abs{X^{\kappa_n}-X}(r,y,\o),\\
I_{14}(n) = & \int\limits^T_0 \d r \int\limits_\R \d y \kappa_n
 \abs{X(r,y,\o)}  \vert \varphi(r,y) \vert.
\end{align*}
By Cauchy-Schwarz, $I_{11}^2(n)$ is bounded by
\begin{equation*}
\norm{\varphi}^2_{L^2([0,T]\times\R)} \int\limits^T_0 \d r \int\limits_\R \(\psi \(X^{\kappa_n}(r,y,\o)\)-\psi\(X(r,y,\o)\)\)^2 \d y.
\end{equation*}
This converges to zero according to Remark \ref{R7.5} 2),
after extracting a further subsequence (not depending on $\o$). 
The square of $I_{12}(n)$ is bounded by
\begin{equation*}
\norm{\varphi}^2_{L^2([0,T]\times\R)} \int\limits_{[0,T]\times\R}
\d r\d y \Phi^4 \(X(r,y,\o)\)\abs{X^{\kappa_n}-X}^2(r,y,\o).
\end{equation*}
This goes to zero because of \eqref{e7.8} in Remark \ref{R7.5} 4).\\
$I^2_{13}(n)$ is bounded by
\begin{equation*}
\kappa_n \norm{\varphi}^2_{L^2([0,T]\times\R)}
 \kappa_n \int\limits_{[0,T]\times\R}\d r\d y \abs{X^{\kappa_n}-X}^2(r,y,\o).
\end{equation*}
After extracting a subsequence, previous expression
 converges to zero because of Lemma \ref{L7.4} c). Finally $I_{14}(n) \substack{\longrightarrow\\n\to \9}0$ by Cauchy-Schwarz and the fact that $\int\limits_{[0,T]\times\R}\d r\d y X^2 (r,y,\o)<\9 \; \pas$\\
This establishes the proof of Lemma \ref{L7.7}.
\end{proof}
\medskip
\begin{description}
\item{4)}
We go on 
with the proof of Theorem \ref{T73}. 
We want to prove that $Y^\infty$ is a weak-strong solution of 
 \begin{equation*}
 Y_t = Y_0 + \int\limits^t_0 \Phi \(X(s,Y_s,\;\cdot\;)\)\d B_s.
 \end{equation*}
According to
  Remark \ref{R2.10} d) item 1) of Definition \ref{DWeakStrong}, it is
 enough to show that for 
$\pas \o$ $Y:=Y^\9(\;\cdot\;,\o)$ is a solution of the following (local) martingale problem. For every $f\in C^{1,2}([0,T]\times \R)$ with compact support, the process
\begin{eqnarray}\label{e7.18}
Z_t^f&:=&f(t,Y_t)-f(0,Y_0)-\frac{1}{2} \int\limits^t_0 f'' (Y_s) \Phi^2\(X(s,Y,\o)\)\d s \nonumber\\
\\
&-& \in_0ç t \partial_s f(s,Y_s)
\end{eqnarray}
is a (local) martingale  under $Q^\o$.\\
For this it is enough to prove that 
under $Q^\o,\; Y$ is a local martingale with quadratic variation $[Y^\9(\;\cdot\;,\o)]_t=\int\limits^t_0 \Phi^2\(X(s,Y_s^\9 (\;\cdot\;,\o)\)\d s$.
According to Proposition \ref{A1} of the Appendix, it is enough to show that
\begin{align}\label{e7.19}
\int\limits^T_0 \d t \abs{\Phi_{\kappa_n}^2\(X(s,Y^n_s(\;\cdot\;,\o)\)-\Phi^2\(X(s,Y^\9_s(\;\cdot\;,\o)\)}\d s\\
\substack{\longrightarrow\\n\to \9}0,
\end{align}
in $Q^\o$ probability. Now, the expectation of \eqref{e7.19} is bounded by
\begin{align*}
& 2\norm{\Phi}_\9 E^{Q^\o}\(\int\limits^T_0 \d s \abs{\Phi_{\kappa_n}
\(X(s,Y^n_s(\;\cdot\;,\o)\)-\Phi\(X(s,Y^\9_s(\;\cdot\;,\o)\)}\)\\
\leq & 2\norm{\Phi}_\9 \left\{ E^{Q^\o}\(\int\limits^T_0 \d s 
\abs{\Phi_{\kappa_n}\(X(s,Y^n_s(\;\cdot\;,\o)\)-\Phi\(X(s,Y^\9_s(\;\cdot\;,\o)\)}^2\)\right\}^{\frac{1}{2}}
\end{align*}
This converges to zero by Lemma \ref{L7.7}.\\
This proves \eqref{e7.19} and that $Y^\9(\;\cdot\;,\o)$ is a weak solution of \eqref{e7.16}.
\end{description}
\begin{description}
\item{5)} The solution $Y^\9(\;\cdot\;,\o)$ of \eqref{e7.15} lives
 on a space $\O_1^\o \times \O$ where, 
$(\O_1^\o,\shh^\o, Q^\o)$ is a probability space  depending on $\o$.\\
We choose now  $\O_1 = C\([0,T];\; \R\) \times \R \times
  C\([0,T];\; \R\)$ and we select
 $Q^{\omega} := Q(\;\cdot\;,\o)$ being the law of
 $Y^\9(\;\cdot\;,\o)$ on $\O_1$. Again 
we set $Y_t(\o_1,\o)= \omega_1^0(t) + a$, this time with 
 $\omega_1 = (\omega_1^0,a,\omega_1^1)$.\\
We have to show that $Y$  is a weak-strong solution of
 \begin{equation*}
 Y_t = Y_0 + \int\limits^t_0 \Phi \(X(s,Y_s,\;\cdot\;)\)\d B_s.
 \end{equation*}
For the moment we have shown that
$Y_t(\cdot,\omega) - Y_0(\cdot,\omega)$ is a martingale 
under $Q^\omega$ for almost all $\omega$
with  $Q^\omega$-quadratic variation given by 
$\int_0^t \Phi^2(X(s, Y(\cdot,\omega)_s) \d s.$
We need to construct a process $B$ on $\Omega \times \Omega_1$,
such that for almost all $\omega$, $B$ is a 
 $Q^\omega$-Brownian motion and
\eqref{DWS} holds for $\gamma(t,\cdot, \omega) = \Phi(X(t,\cdot,\omega))$.
Let $\beta_t(\omega_1)  =   \omega_1^1(t),$ a supplementary Brownian motion
on $\Omega_1$ which is $Q^\omega$-independent on $Y$ and
we remind that $Y_t(\omega_1) = \omega_1^0(t) + a$.
$\beta$ can be also considered as a Brownian motioon on $\Omega \times 
\Omega_1$ which is $Q$-independent of $Y$ and $(\shf_t)$.\\
We define
$$ B_t(\cdot,\omega) = \int_0^t \d Y_s(\cdot,\omega) 
1_{\{\gamma(s,\xi,\omega) \neq 0\}} \frac{1}{\gamma(s,\xi,\omega)}
+ \int_0^t   1_{\{\gamma(s,\xi,\omega) = 0\}} \d \beta_s.$$ 
 Now for $Q^\omega$-a.s. the quadratic variation of 
the $Q^\omega$-martingale $B(\cdot,\omega)$
is $t$, so that, by L\'evy characterization theorem,
 $B(\cdot, \omega)$ is a Brownian motion under $Q^\omega$.

It remains to show  items 2) and 3) of the definition  of  weak-strong solution.
Let $(\shy_t)$ be the canonical filtration of the process 
$Y(\cdot, \omega)$. Item 3) follows because of item 1)
and because $\gamma(t,\cdot, \omega) = \Phi(X(t,\cdot,\omega))$
is progressively measurable.
Concerning item 2) we see that 
under $Q$ defined by $P$ and the kernel $Q(\cdot,\omega)$,
$W^1, \ldots,W^N$ are $Q$-martingales with $(\shg_t)$
as defined in Definition \ref{DWeakStrong}.
Indeed let $F$ be a bounded $\shf_s$-measurable random variable and $A$ be a
bounded $\shy_s$-measurable r.v. 
Let $1 \le i \le N$. 
By item 3) $E^{Q^{\omega}} (A)$ is $\shf_s$-mesurable, so 
$$ E^Q((W^i_t - W^i_s) F A) =  E^P((W^i_t - W^i_s) F 
E^{Q^{\omega}} (A) ) =  0, $$
since $W^i$ is an  $\shf_s$-martingale.


\item{6)}
 The final step consists in proving that $X$ is the family of $\mu$-marginal laws of $Y$, under $Q$ defined by
  $\int_\O Q(\cdot\;,\o) dP(\o) $. Let $\o\in\O$ outside some $P$-null set.\\
By step 1) of the proof of this Theorem \ref{T73},
we know that $X^\kappa$ fulfills, for almost all $\o$,
$$
\int\limits_\R \d \xi X^\kappa (t,\xi)\varphi(\xi)
=  E^{Q^\kappa (\;\cdot\;,\o)} \(\varphi(Y_t)\she_t \(\int\limits^\cdot_0
 \mu (\d s,Y_s)\)\),
$$
for all $\varphi\in\shs(\R)$.
We recall that, according to the lines before step 3) of this proof,
after use of Skorohod theorem and a change of 
probability space (depending on $\o$), there is a sequence $(\kappa_n)$
and processes $Y^n(\cdot, \o)$ converging ucp to $Y^\infty(\cdot, \o)$
such that
\begin{equation}\label{e7.20}
\int\limits_\R \d \xi X^{\kappa_n} (t,\xi)\varphi(\xi)=E^{Q^\o}
 \(\varphi(Y^n)\she_t \(\int\limits^\cdot_0 \mu (\d s,Y^n)\)\),
\end{equation}
for every $\varphi\in\shs(\R)$.
It remains to show, for every $t\in[0,T],\; \varphi\in \shs(\R)$, that
\begin{equation}\label{e7.21}
\int\limits_\R \d \xi\varphi(\xi) X(t,\xi,\o) = E^{Q^\o} \(\varphi(Y^\9_t)\she_t
 \(\int\limits^\cdot_0 \mu (\d s,Y^\9_s)\)\).
\end{equation}
Let $\o\in\O$ outside a $P$-null set.\\
Since $t \mapsto X(t,\;\cdot\;)$ is continuous from $[0,T]$ to 
$\shs'(\R)$ and the right-hand side is continuous on $[0,T]$ for fixed $\varphi\in\shs(\R)$, it is enough to show \eqref{e7.21} for almost all $t\in[0,T]$.\\
Now for almost all $t$, the left-hand side of \eqref{e7.21} is approached by the left-hand side of \eqref{e7.20}. It remains to show that the right-hand side of \eqref{e7.21} is the limit of the right-hand side of \eqref{e7.20}.
\end{description}

According to Proposition \ref{A3} in the Appendix it will be enough to show the following.
\begin{enumerate}[i)]
\item
\begin{equation}\label{e7.22}
\sum^N_{i=1} \int\limits^t_0 e^i (Y^n_s) \d W^i_s -\frac{1}{2} \int\limits^t_0 e^i (Y^n_s)^2 \d s
\end{equation}
converges in law (with respect to $Q^\o$) to
\begin{equation}\label{e7.23}
\sum^N_{i=1} \(\int\limits^t_0 e^i (Y_s) \d W^i_s -\frac{1}{2} \int\limits^t_0 e^i (Y_s)^2 \d s\).
\end{equation}
\item $\varphi(Y^n_t) \she \(\int\limits^\cdot_0 \mu (\d s,Y^n_s)\) $ is a sequence which is uniformly integrable. 
\end{enumerate}
 We check now those properties.\\
i) \eqref{e7.22} equals
\begin{equation*}
J_1(n) + J_2(n)
\end{equation*}
where
\begin{align*}
J_1(n) = & \sum^N_{i=1} \left\{W^i_s e^i (Y^n_s)-\frac{1}{2} \int\limits^t_0 e^i (Y^n_s)^2 \d s\right.\\
& \left. -\frac{1}{2} \int\limits^t_0 W^i_s (e^i)'' (Y^n_s) \Phi^2_{\kappa_n} \(X^{\kappa_n} (s,Y^n_s,\o)\)\d s\right\}\\
J_2(n) = & - \sum^N_{i=1} \int\limits^t_0 W^i_s (e^i)' (Y^n_s) \d Y^n_s.
\end{align*}
Lemma \ref{L7.7} implies that
\begin{align*}
\int\limits^T_0 \abs{\Phi_{\kappa_n}^2\(X^{\kappa_n}(s,Y^n_s,\o)\)-\Phi^2\(X(s,Y_s,\o)\)}\d s \underbrace{\longrightarrow}_{n \rightarrow \infty}
0
\end{align*}
in probability. Consequently, since $Y^n\to Y$ ucp, it follows that
$J_1(n)$ converges in probability to
\begin{align*}
J_1  := & \sum^N_{i=1} \(W^i_t e^i (Y_t)-\frac{1}{2} \int\limits^t_0 e^i 
(Y_s)^2 \d s\right.\\
& -\frac{1}{2} \int\limits^t_0 W^i_s (e^i)'' (Y_s) \Phi^2 \(X (s,Y_s,\o)\)\d s.\\
\end{align*}
According to \cite{pages},
$J_2(n) \to J_2$
 in law,
where $J_2 =  \int\limits^t_0 W^i_s (e^i)' (Y_s)\d Y_s$. 
 This implies that \eqref{e7.22} converges in law to \eqref{e7.23} and item i) is established.\\
To prove ii) we only need to prove that
\begin{equation*}
\sup_n E^{Q^\o} \(\she \(\int\limits^t_0 \mu (\d s,Y_s^n)\)^2\)<\9.
\end{equation*}
The integrand of previous expectation gives
\begin{equation*}
\exp \( 2(J_1(n)+J_2(n))\).
\end{equation*}
For each $\o,\; \exp \( 2(J_1(n)\) $ is bounded, so it remains to prove that,
 for every $0\leq i\leq N$
\begin{equation}\label{e7.24}
\sup_n E^{Q^\o} \(\exp \( -2\int\limits^t_0 W^i_s (e^i)' (Y^n_s)\d Y^n_s\)\)<\9.
\end{equation}
Since $-2\int\limits^t_s W^i_s (e^i)' (Y^n_s)\d Y^n_s $ is a $Q^\o$-martingale,
\begin{equation*}
\she_t^n :=  \exp\( -2\int\limits^t_0 W^i_s (e^i)' (Y^n_s)\d Y^n_s
 - 2\int\limits^t_0 (W^{i})^2_s (e^i)^{'2} (Y^n_s)
\Phi^2_{\kappa_n}\(X^{\kappa_n}(s,Y_s^n,\o)\)\d s\)
\end{equation*}
is an (exponential) martingale.
Consequently \eqref{e7.24} is bounded by
\begin{align*}
& \sup_n E^{Q^\o} \(\she_t^n \exp \(2\int\limits^t_0 (W^{i})^2_s (e^{i'})^{2}
 (Y^n_s)
 \Phi^2_{\kappa_n}(X^{\kappa_n}(s,Y_s^n,\o))\d s\) \)\\
\leq & \exp \(2\norm{e^{i'}}^2_\9 \(\norm{\Phi}^2_\9 + \kappa_n\) 
\int\limits^T_0 (W^{i}_s)^2 \d s\).
\end{align*}
This quantity is bounded for each $\o$, ii) is now established and so is
 the step 6) of Theorem \ref{T73}.
\end{proof}

\begin{appendix}
\section{Technicalities}

\setcounter{equation}{0}

\begin{prop} \label{A1}
Let $(\O,\shh,Q_1)$ be a probability space. Let $(Y^n)$ be a sequence of continuous local martingales such that $Y^n \to Y$ ucp; we suppose the existence of a
  adapted continuous process $A$, adapted to the canonical filtration
associated with $Y$ such that the total variation 
of $[Y^n]_t-A_t$ goes to zero in probability for every $t\in [0,T]$.\\
Then $Y$ is a local martingale whose quadratic variation is $A$.
\end{prop}
\begin{rem} \label{A2}
\begin{enumerate}[i)]
\item A local martingale is in particular a local martingale with respect to its own filtration.
\item 
This proposition should be known in the literature but for the moment
 we cannot find the reference.
The difficulty is that there is no filtration specified.
\end{enumerate}
\end{rem}
\begin{proof}[Proof of Proposition \ref{A1}]

After a localization procedure, we can suppose that
the process $Y$ is bounded and $Y^n$ are bounded by the same
constant. 
By classical approximation techniques we only need to show that
\begin{equation}\label{eFA2}
\tag{A.1} 
f(Y_t) - \frac{1}{2}  \int^t_0 f''(Y_s) \d s
\end{equation}
is a martingale, for every $f\in C^2$ with compact support.
Let $0\leq s<t\leq T$ and $\Theta:C([0,s])\to \R$ be a bounded and continuous
 functional.
We need to show that 
\begin{equation}\label{eA3}
\tag{A.2}
E\( (f(Y_t)-f(Y_s)-\frac{1}{2} \int^t_s f''(Y_r) \d A_r) \Theta_s\)=0,
\end{equation}
where $\Theta_s =\Theta \( Y_r:r\leq s\right)$. The left-hand side of
 \eqref{eA3} equals
\begin{equation*}
I_1(n)+ I_2(n),
\end{equation*}
where
\begin{equation*}
I_1(n) = E\( \(f(Y_t)-f(Y_s)\) \Theta_s -\(f(Y_t^n)-f(Y^n_s) \)\Theta^n_s\)
\end{equation*}
and
\begin{equation*}
\Theta^n_s = \Theta (Y^n_r:r\leq s),\; \Theta_s = \Theta (Y_r:r\leq s),
\end{equation*}
\begin{align*}
I_2(n) & = E\( \(f(Y_t^n)-f(Y_s^n) -\frac12 \int^t_s f''(Y_r^n) \d [Y^n]_r \) \Theta^n_s\),\\
I_3(n) & = E\( \frac12 \int^t_s f''(Y_r^n) \d [Y^n]_r 
\(\Theta_s^n - \Theta_s\)\)\\
I_4(n) & = E\( \left\{ \frac12 \int^t_s (f''(Y_r^n) - f''(Y_r))
\d [Y^n]_r\right\} 
\Theta_s\),\\
I_5(n) & = E\( \left\{ \frac12 \int^t_s (f''(Y_r))\d ([Y^n]-A_r)\right\} 
\Theta_s\).
\end{align*}
Observe that $I_2(n)=0$ since $Y^n$ is a martingale.\\
Taking into account the ucp convergence of $Y^n$
to $Y$, it is not difficult to show that
\begin{equation*}
I_i(n)\to 0,\; i=1,3.
\end{equation*}
By BDG inequality there is a constant $C>0$ such that
\begin{equation} \label{BDG1}
\tag{A.3}
 E([Y^n]_T) \le C E( \sup_{t \le [0,T]} \vert Y^n_t \vert).
\end{equation}
As far as $I_5(n)$ is concerned, we have 
$[Y^n] \rightarrow A $ ucp, taking into account the fact that
$A$ is a continuous process, the convergence in probability
and a Dini type argument, see e.g. Lemma 3.1 of \cite{rv4}.
So after extraction of subsequences we get
\begin{equation} \label{E38}
\tag{A.4}
\int_s^t Z_r (\d [Y^n] - \d A)_r \rightarrow 0,
\end{equation}
in probability, for the continuous process $Z_r = f''(Y_r)$.
On the other hand
$$ E( \int_s^t Z_r \d [Y^n]_r)^2 \le C \Vert f'' \Vert_\infty 
   E([Y^n]T),$$
is uniformly bounded by \eqref{BDG1} and so the family of r.v. in \eqref{E38}
is uniformly integrable.
Finally \eqref{E38} also holds in $L^1$. \\
Concerning $I_4(n)$, since $f''(Y^n)$ converges to $f''(Y)$
a.s. uniformly, by \eqref{BDG1}, the integral inside the expectation 
converges to zero a.s.  Then, by Lebesgue dominated convergence theorem
its expectation, i.e.
$I_4(n)$ converges to zero.


Finally the result follows.
\end{proof}
\begin{prop} \label{A3}
Let $(Z_n)$ be a sequence of random elements with values in a Banach
 space $E$ converging in law to some random element  $Z$ 
still with values in $E$.
Let $\psi:E\to \R$ continuous, such that $(\psi(Z_n))$ are uniformly
 integrable. Then $\lim\limits_{n\to \infty} E\(\psi (Z_n)\right) = E\(\psi (Z)\right)$.
\end{prop}
\begin{proof}[Proof of Proposition \ref{A3}]
According to Skorohod theorem, there is a new probability space 
$(\wt \O,\wt \shf,\wt P)$, random elements
$ \wt Z_n (\text {resp. }\wt Z)$ on $(\wt \O,\wt \shf,\wt P)$ with
 the same distribution as $Z_n (\text{resp. }Z)$ and $\wt Z_n \to \wt Z$ 
a.s. So $\psi(\wt Z_n)\to \psi (\wt Z)$ a.s.
Clearly the sequence $\psi(\wt Z_n)$ is also uniformly integrable; finally $\psi(\wt Z_n)\to \psi(\wt Z)$ in $L^1(\wt P)$. In particular $E^{\wt P}(\psi(\wt Z_n)) \substack{\to\\n\to \9} E^{\wt P}(\psi(\wt Z))$, and so the result follows.
\end{proof}
\begin{prop} \label{A4}
Let $Y_0$ be distributed according to $x_0$. Let $a:[0,T]\times\R\to\R$ be a Borel function such there are $0<c<C$ with
\begin{equation}
\label{eFA4}
\tag{A.5}
c \leq a(s,\xi)\leq C, \quad \forall \; (s,\xi)\in[0,T]\times \R.
\end{equation}
We fix $0 \le r  \le t \le T$.
We set $a_n(t,x)=\int\limits_\R \rho_n (x-y) a(t,y)\d y$ where $(\rho_n)$
 is the usual sequence of mollifiers converging to the Dirac delta.
The unique solutions $S^n$ to
\begin{equation}\label{eFA5}
S^n_t=Y_0 + \int^t_r a_n (s, S^n_s)\d B_s,
\end{equation}
$B$ being a classical Wiener process, converges in law to the (weak unique solution) of
\begin{equation}\label{eFA6}
S_t=Y_0+\int^t_r a (s, S_s)\d B_s,
\end{equation}
\end{prop}
\begin{proof}
\begin{enumerate}[i)]
\item Let $r\in[0,T[,\; y\in \R$. According to Problem 7.3.3 of \cite{sv}, the equation
\begin{equation}\label{eFA7}
\tag{A.6}
S_t=y+\int^t_r a(s,S_s)\d B_s,
\end{equation}
admits a solution, which is unique in law. We denote by $P^{r,y}$ the law on $C([r,T])$ of the corresponding canonical process. The equation
\begin{equation}\label{eFA8}
S_t=y+\int^t_r a^n (s,S_s)\d B_s
\end{equation}
admits even a strong solution since $a^n$ is Lipschitz with linear growth. We denote with 
$P_n^{r,y}$ the corresponding law. Moreover, those processes are
 Markovian.
\item By the same mentioned problem 7.3.3 of \cite{sv}, there is a constant $C_1$ only depending on $c,C$ in \eqref{eFA4} and $T$ such that
\begin{equation}\label{eFA9bis}
\tag{A.7}
E^{P_n^{r,y}} \( \int_{[r,T]\times \R} f(r,S_r) \d r\) \leq C_1 \norm{f}_{L^2([0,T]\times \R)},
\end{equation}
for every $n\in \N \cup \{\9\},(r,y)\in [0,T[\times \R$, and every bounded
function $f:[0,T]\times \R \to\R$ with compact support.
For convenience in \eqref{eFA9bis} we set $P^{r,y}:=P_\infty^{r,y}$.
\item From ii), there are Borel functions $q_n:[0,T]^2\times \R \to \R, n\in \N \cup \{\9\}$ such that $(t,z)\to q_n(r,t,y,z)\in L^2([r,T]\times\R)$
 for every $(r,y)\in [0,T]\times \R$ and the law of $S_t$ under $P^{r,y}_n$ equals $q_n(r,t;y,\;\cdot\;) \ dt$ \ a.e.
Moreover
\begin{equation*}
\int_{[r,T]} q^2_n (r,t;y.z)\d t \d z \leq C_1^2.
\end{equation*}
\item For $0\leq r < t \leq T,\; y\in \R$ the laws $\(P^{r,y}_n \)$ are tight.
In fact, by Burkholder-Davies-Gurdy inequality, there are constants $C_2,C_3 >0$ such that
\begin{align*}
E^{P_n^{r,y}}\((S_t-S_r)^4\)\leq & \\
C_2 E^{P_n^{r,y}}\(\(\int^t_r a^2_n(s,S_s)\d s\)^2\) \leq & C_3 (t-r)^2.
\end{align*}
A slight adaptation of Problem 4.11 associated with Theorem 4.10 Chapter 2 of \cite{KARSH} implies the tightness.
\item In particular, for each subsequence there is a subsubsequence converging weakly.
\item In fact for every $0\leq r<t\leq T,\; y\in\R$, the sequence $\(P_n^{r,y}\)$ converges weakly to $P^{r,y}$.\\
To prove this, by point v) and the uniqueness in law of \eqref{eFA7}, it is enough to show that the limit of a weakly converging subsequence of $\(P_n^{r,y}\)$ (still denoted in the same manner) fulfills the martingale problem related to \eqref{eFA7}.
Let $Q^{r,y}$ be such a limit. Using the Markov property we only need to show that for every $f \in C^\9_0 ([0,T] \times \R)$
\begin{equation}\label{eFA9}
E^{Q^{r,y}}\(f(t,S_t) - f(0,y)-\int^t_r  \partial_{xx}f(s,S_s)\frac{a^2}{2}
 (s,S_s)\d s\) = 0
\end{equation}
For this, taking into account the fact that $P_n^{r,y} \to Q^{r,y}$ and the fact that
\begin{equation}\label{eFA10}
E^{P_n^{r,y}} \(f(t,S_t) - f(0,y)-\int^t_r \partial_{xx}f(s,S_s)\frac{a_n^2}{2}
 (s,S_s)\d s\),
\end{equation}
it will be enough to show that
\begin{align}\label{eFA11}
\tag{A.8}
E^{P_n^{r,y}} \(\int^t_r\partial_{xx}f  (s,S_s)\frac{a_n^2}{2} (s,S_s)\d s\),\\
\substack{\rightarrow\\ n\to \9}E^{Q^{r,y}}\( \int^t_r \partial_{xx}f(s,S_s)
\frac{a^2}{2} (s,S_s)\d s\).
\end{align}
The left hand side of \eqref{eFA11} equals $I_1(n)+I_2(n)$ where
\begin{align*}
I_1(n) & = E^{P_n^{r,y}} \(\int^t_r \partial_{xx}f(s, S_s)\frac{a_n^2-a^2}{2} (s,S_s)\d s\)\\
I_2(n) & = E^{P_n^{r,y}} \(\int^t_r\partial_{xx}f (s,S_s)\frac{a^2}{2} (s,S_s)\d s\)\\
& - E^{Q^{r,y}} \(\int^t_r  \partial_{xx}f(s,S_s)\frac{a^2}{2} (s,S_s)\d s\).
\end{align*}
By item iii) $I_1(n)$ gives
\begin{equation}\label{eFA12}
\int_{[r,T]\times \R} \d t \d z \partial_{xx}f (t,z) \frac{a_n^2-a^2}{2}(t,z)q(r,t;y,z).
\end{equation}
Since $ \partial_{xx}f$ has compact support, $\abs{a_n^2-a^2}\leq 2 C_1$, together with $a_n^2 \to a^2 \d t \d z$ a.e.,
 Cauchy-Schwarz and Lebesgue dominated convergence theorem imply that \eqref{eFA2} goes to zero when $n\to \9$.
Concerning $I_2(n)$, we only need to prove that for every
 $f:[0,T]\times \R\to\R$ bounded measurable with compact support verifies
\begin{equation}\label{eFA13}
\tag{A.9}
\lim_{n\to\9} E^{P_n^{r,y}} \(\int_r^T f (s,S_s)\d s\) = E^{Q^{r,y}}
 \(\int^T_r f (s,S_s)\d s\).
\end{equation}
Indeed we can prove \eqref{eFA13} holds for $f \in L^2([0,T] \times \R)$.
In fact, by the convergence in law, \eqref{eFA13} holds for every $f\in C^{0}(r,T)\times\R)$
with compact support.
\eqref{eFA9bis} and Banach-Steinhaus allow to establish \eqref{eFA13} and therefore the conclusion.
\end{enumerate}
\end{proof}

\section{Uniqueness for the porous media equation 
with  noise}

\setcounter{equation}{0}

We state here a general uniqueness lemma when the coefficient
$\psi: \R \rightarrow \R$ is Lipschitz.
Since the paper concerns one-space dimension porous media
type equation, we remain in that framework.
However, the theorem below easily extends to
the multi-dimensional case.

We consider here an infinite
nomber of modes for the random field $\mu$, i.e.
 $\mu(t,\xi)=\sum\limits^{\9}_{i=0} e^i(\xi)W^i_t$ where $W^i$ are
 independent Brownian motions, $e^i :\R^d\to\R\; \in 
L^1(\R^d)$ being $H^{-1}$ multipliers with norm $\shc(e^i)$,
$W^0_t  = t$.
\begin{theo} \label{TB1}  Let $x_0 \in \shs'(\R^d)$ and
 make the following assumptions.
\begin{enumerate}[i)]
\item $\psi$ is Lipschitz,
\item 
$ \sum_{i=1}^\infty  \left(\shc(e_i)^2 
 +  \Vert e^i\Vert_\infty^2 \right) < \infty.$
\end{enumerate}
Then equation \eqref{PME} admits at most one
 solution among the random fields $X:]0,T]\times\R\times\O\to\R$ such that 
\begin{equation}\label{eFB1}
\tag{B.1}
\int_{[0,T]\times\R} X^2(s,\xi)\d s\d \xi <\9 \quad \text{a.s.}
\end{equation}
\end{theo}
\begin{rem}\label{remB2bis} 
We observe that condition ii) is compatible with (3.1) of \cite{BRR3}.
\end{rem}

\begin{rem}\label{remB2} Let $X$ be a solution of \eqref{PME}.
\begin{enumerate}[i)]
\item There is  a $P$ null set $N_0$, so that for $\o \not \in N_0,\;
 X(t,\;\cdot\;)\in L^2(\R)$ for almost all $t\in[0,T]$.
\item Condition \eqref{eFB1} also implies that
\begin{equation}\label{eFB2}
\int^T_0 \norm{X(s,\;\cdot\;)}^2_{H^{-1}} \d s <\9 \quad \text{a.s.}
\end{equation}
\item Since $\psi$ is Lipschitz and $\psi(0) = 0$, \eqref{eFB1} implies that
$\int_0^T\Vert \psi(X(r,\cdot))\Vert_{L^2}^2 \ dr < \infty$. So  
$\int^t_0 \d s \psi (X(s,\cdot))$ is a Bochner integral with values in 
$L^2(\R)$.
\item Consequently, 
 $t\mapsto \Delta \( \int\limits^t_0 \d s \psi (X(s,\;\cdot\;))\)$
 is continuous from $[0,T]$ to  $H^{-2}(\R)$ and so also in
 $\shs'(\R)$;
since $e^i, 1 \le i \le N,$ are $H^{-1}$-multipliers, then
$t\mapsto \int\limits^t_0 \mu (\d s, \cdot) X(s,\cdot)$ belongs to $C\left([0,T]; H^{-1}(\R) \right)$;
since $x_0 \in \shs'(\R)$  and $X \in C([0,T]; \shs'(\R))$ a.s.,
 it follows that
for $\o$ not belonging to a null set, we have
\begin{equation}\label{eFB3}
X(t,\;\cdot\;)=x_0+\Delta \(\int\limits^t_0 \psi (X(s,\;\cdot\;))\d s \right)
 + \int\limits^t_0 \mu ( \d s , \cdot) X(s,\cdot) \quad t\in [0,T],
\end{equation}
as an identity in $\shs'(\R)$.
\item If $x_0 \in H^{-1}$ then $X\in C\([0,T];H^{-2}\)$, for $\o \not \in N_0$,
$N_0$ a $P$-null set.
 Moreover, if $x_0 \in L^2$ or $\psi$   is non-degenerate, then, 
by Theorem 3.4 of \cite{BRR3}, 
then  $t \mapsto  \int_0^t \psi(X(s,\cdot)) ds \in  C([0,T]; H^1(\R))$. 
\item If $x_0 \in H^{-s}$ for some $s \ge 2$ then $X\in C\([0,T];H^{-s}\)$, 
for $\o \not \in N_0$,
$N_0$ a $P$-null set.
\item Let $\varepsilon >0$ and consider a sequence of mollifiers $(\phi_\varepsilon)$ converging to the Dirac measure. Then $X^\varepsilon
 (t,\;\cdot\;) = X(t,\cdot) \star \phi_\varepsilon$ belongs a.s. 
to $C\([0,T];L^2(\R)\)$.
\item All the previous items hold provided there exists a solution $X$ verifying \eqref{eFB1}.
\item Since $\psi$ is Lipschitz, there is $\alpha>0$ such that
\begin{equation*}
\(\psi(r)-\psi(\bar r)\) (r-\bar r)\geq \alpha \( \psi (r)-\psi(\bar r)\)^2.
\end{equation*}
\end{enumerate}
\end{rem}

\begin{proof}[Proof]
Let $(\phi_\varepsilon, \varepsilon > 0)$ be a sequence of mollifiers as 
in Remark \ref{remB2} 
vii).
Let $X^1,X^2$ be two solutions of \eqref{PME}.
For $i = 1,2$,  we set 
$(X^i)^\varepsilon(t,\cdot) = X^i(t,\cdot) \star \phi_\varepsilon.$
We denote $X=X_1 - X_2$ and   $X^\varepsilon=(X^1)^\varepsilon-(X^2)^\varepsilon$
 which a.s. belongs to $C\([0,T],\;L^2(\R)\) \subset C\([0,T];\;H^{-1}\)$.
We expand
\begin{align*}
g_\varepsilon(t) := & \norm{X^\varepsilon(t,\;\cdot\;)}^2_{H^{-1}}\\
& = \int_\R \((I-\Delta)^{-1} X^\varepsilon (t,\;\cdot\;) \) (\xi) X^\varepsilon (t,\xi) \d \xi.
\end{align*}
\end{proof}
It\^o formula gives
\begin{align}\label{eFB4}
\tag{B.2}
g_\varepsilon(t) = & 2\int^t_s < X^\varepsilon (s,\;\cdot\;), X^\varepsilon (\d s,\;\cdot\;) >_{H^{-1}}\\
& + \sum^\9_{i=1} \int^t_s \d s \norm{(e^i X)(s,\cdot)
\star \phi_\varepsilon}^2_{H^{-1}}.
\end{align}
On the other hand we have
\begin{align}\label{eFB5}
\tag{B.3}
X^\varepsilon (t,\;\cdot\;) = & \int^t_0 \d s \Delta \left[ \left\{\psi
 (X^1(s,\;\cdot\;))-\psi (X^2(s,\;\cdot\;))\right\}\star \phi_\varepsilon \right]\\
& + \int^t_0 (\mu(\d s,\;\cdot\;)X)\star \phi_\varepsilon.
\end{align}
\begin{align}\label{eFB6}
\tag{B.4}
(I-\Delta)^{-1} X^\varepsilon (t,\;\cdot\;) & = -\int^t_0 \(\psi(X^1(s,\;\cdot\;))-\psi(X^2(s,\;\cdot\;))\)\star \phi_\varepsilon\\
& + \int^t_0 \d s(I-\Delta)^{-1}\(\psi(X^1(s,\;\cdot\;))-\psi(X^2(s,\;\cdot\;))\)\star \phi_\varepsilon\\
& + \sum^\9_{i=0} \int^t_0 \d W^i_s \left[(I-\Delta)^{-1}(e^i X(s,\;\cdot\;))\right]\star \phi_\varepsilon.
\end{align}
We define 
\begin{equation*}
M_t = \sum^\9_{i=1} \int^t_0 < (I-\Delta)^{-1} X (s,\;\cdot\;),
 e^i X(s,\;\cdot\;) >_{L^2} dW^i_s.
\end{equation*}
We observe that previous $M$ is well-defined and it is
a local martingale. Indeed, by Remark \ref{remB2} iv), 
$X \in C([0,T];H^{-2})$, so by   similar arguments as in \eqref{e5.6bis},
\begin{eqnarray} \label{EFB100}
 \sum^\9_{i=1} \int^t_0 < (I-\Delta)^{-1} X (s,\;\cdot\;),
 e^i X(s,\;\cdot\;) >_{L^2}^2 ds &\le& 
\sup_{s \in [0,T]} \Vert X(s,\cdot) \Vert_{H^{-2}}^2 
\sum_{i = 1}^\infty  \Vert e^i\Vert_\infty^2 \\
&&\int_0^T \Vert X (s,\;\cdot\;)\Vert_{L^2}^2  \d s < \infty.
\end{eqnarray}
Using \eqref{eFB4}, \eqref{eFB5} and \eqref{eFB6} we get
\begin{align}\label{eFB7}
\tag{B.5}
g_\varepsilon (t) = & \sum^\9_{i=1}\int^t_0 \d s \norm{(e^i X)(s,\cdot)\star
 \phi_\varepsilon}^2_{H^{-1}}\\
& - 2\int^t_0 < X^\varepsilon (s,\;\cdot\;), \left[\psi (X^1(s,\;\cdot\;))-\psi 
(X^2(s,\;\cdot\;))\right]\star \phi_\varepsilon >_{L^2}\\
& + 2\int^t_0 < X^\varepsilon (s,\;\cdot\;), (I-\Delta)^{-1}\left[\psi
 (X^1(s,\;\cdot\;))-\psi (X^2(s,\;\cdot\;))\right]\star \phi_\varepsilon >_{L^2}\\
& + 2\int^t_0 < X^\varepsilon (s,\;\cdot\;), (I-\Delta)^{-1} \left[e^0 
X(s,\;\cdot\;)\right]\star \phi_\varepsilon >_{L^2}\\
& + M_t^\varepsilon,
\end{align}
where $M^\varepsilon$ is the local martingale defined by
\begin{equation*}
M^\varepsilon_t = \sum^\9_{i=1} \int^t_0 < X^\varepsilon (s,\;\cdot\;), 
(I-\Delta)^{-1}| \(e^i X(s,\;\cdot\;)\) \star \phi_\varepsilon >_{L^2} dW^i_s,
\end{equation*}
which is again well-defined by similar arguments
as for the proof of \eqref{EFB100}.
Taking into account \eqref{eFB1} and the Lipschitz property for $\psi$, we can take the limit when $\varepsilon \to 0$ in \eqref{eFB7} and for $g(t):=\norm{X(t,\;\cdot\;)}^2_{H^{-1}}$, we obtain
\begin{align}\label{eFB8}
\tag{B.6}
g(t) & + 2 \int^t_0 \d s \< X(s,\;\cdot\;),\; \psi \(  X^1(s,\;\cdot\;)\)-\psi \(  X^2(s,\;\cdot\;)\)\>_{L^2}\\
= & \sum^\9_{i=1} \int^t_0 \d s \norm{e^i X(s,\;\cdot\;)}^2_{H^{-1}}\\
& + 2 \int^t_0 \d s \< (I-\Delta)^{-1}X(s,\;\cdot\;),\; \psi \(  X^1(s,\;\cdot\;)\)-\psi \(  X^2(s,\;\cdot\;)\)\>_{L^2}\\
& + 2 \int^t_0 \d s \< X(s,\;\cdot\;)\;e^0 X(s,\;\cdot\;)\>_{H^{-1}} + M_t.
\end{align}
The convergence $M^\varepsilon \to M$ when $\varepsilon \to 0$ holds ucp since
\begin{equation*}
\sum^\9_{i=1} \int^t_0 \d s \big | \< (X^\varepsilon -X)(s,\;\cdot\;),\;
 e^i X(s,\;\cdot\;)\>_{H^{-1}} \big |  ^2 \substack{\longrightarrow\\
 \varepsilon \to 0} 0
\end{equation*}
because of property ii).\\
We take into account the inequality
\begin{equation*}
2ab \leq \frac{a^2}{\alpha}+b^2\alpha,
\end{equation*}
for $a,b\in \R$, $\alpha$ being the constant
 appearing at item vii) of Remark \ref{remB2}. Then the second term 
of the right-hand side of equality \eqref{eFB8} is bounded by
\begin{align*}
& \alpha^2 \int^t_0 \d s \norm{(I-\Delta)^{-1}X(s,\;\cdot\;)}^2_{L^2} + \frac{1}{\alpha^2}\int^t_0 \(\psi \(  X^1(s,\;\cdot\;)\)-\psi \(  X^2(s,\;\cdot\;)\)\)^2_{L^2} \d s\\
\leq & \alpha^2 \int^t_0 \d s \norm{X(s,\;\cdot\;)}^2_{L^2} + \int^t_0 \d s \< \psi \(  X^1(s,\;\cdot\;)\)-\psi \(  X^2(s,\;\cdot\;)\),\; X(s,\;\cdot\;)\>_{L^2}.
\end{align*}
This together with \eqref{eFB8} gives $\pas$
\begin{align}\label{eFB9}
g(t) & +\int^t_0 \d s\< X(s,\;\cdot\;),\;\psi \(  X^1(s,\;\cdot\;)\)-\psi \(  X^2(s,\;\cdot\;)\)\>_{L^2}\\
\leq & 2 \int^t_0 \d s \< X(s,\;\cdot\;)\;e^0 X(s,\;\cdot\;)\>_{H^{-1}}\\
& + \sum^\9_{i=1}\int^t_0 \d s \norm{(e^i X)(s, \cdot)}^2_{H^{-1}}+M_t.\\
\end{align}
Since $e^i,\; i\in \N$ are $H^{-1}$-multipliers and taking into 
account Hypothesis ii), we get
\begin{equation}\label{eFB10}
\tag{B.7}
g(t)\leq M_t + (2 + \sum_{i = 1}^\infty \shc(e_i)^2) \int^t_0 \d s g(s).
\end{equation}
We proceed now by localization.
Let $(\varsigma^\ell)$ be a sequence of stopping times defined by
\begin{equation*}
\varsigma^\ell = \inf \{t\in[0,T]|\int^t_0 \d s \norm{X(s,\;\cdot\;)}^2_{L^2}
 \geq \ell, \Vert X(s,\cdot)\Vert_{H^{-2}} \ge \ell  \},
\end{equation*}
with the convention that $\varsigma^\ell =\9$ if $\{ \;\}=\emptyset$.
Since $\int^T_0\norm{X(s,\;\cdot\;)}^2_{H^{-1}} \d s <\9$ a.s. we have $\O=\bigcup\limits^\9_{\ell=1}\{\varsigma^\ell >T\}$ up to a null set. Clearly the stopped processes $M^\varsigma_\ell$ are martingales starting from zero. We evaluate \eqref{eFB10} at $t\wedge\varsigma^\ell$ and we take the expectation wich gives
\begin{equation*}
E\(g(t\wedge\varsigma_\ell)\)\leq \shc E\(\int_0^{t\wedge\varsigma_\ell} g(s) \d s\).
\end{equation*}
Since $\int\limits_0^{t\wedge\varsigma_\ell} g(s) \d s
\leq \int_0^{\varsigma_\ell} g(s) \d s\leq \ell,\; E\(g(t\wedge\varsigma_\ell)\)$ is finite for every $\ell>0$. Consequently
\begin{equation*}
E\(g(t\wedge\varsigma_\ell)\)\leq \shc \int^t_0 \d s E\(g(s\wedge\varsigma_\ell)\)
\end{equation*}
and by Gronwall lemma it follows
\begin{equation*}
E\(g(t\wedge\varsigma_\ell)\)=0\quad \forall \; \ell \in \N.
\end{equation*}
By \eqref{eFB8} $g$ is a.s. continuous. On the other hand, for every $t\in[0,T],\; \lim\limits_{\ell \to 0} t\wedge\varsigma_\ell =t$ implies that
\begin{align*}
E\(g(t)\)& =E\(\liminf_{\ell \to 0} g(t\wedge\varsigma_\ell)\)\\
& \le \liminf_{\ell \to \infty} E\(g(t\wedge\varsigma_\ell)\)=0
\end{align*}
by Fatou's lemma. This concludes the proof.
\end{appendix}
\bigskip

{\bf ACKNOWLEDGEMENTS} 
\noindent

Financial support through the SFB 701 at Bielefeld University and
NSF-Grant 0606615
is gratefully acknowledged. 
 The third named author was partially supported
by the ANR Project MASTERIE 2010 BLAN 0121 01.
Part of this work was written during a stay of the first
and third named authors at
the Bernoulli Center (EPFL Lausanne). 

\addcontentsline{toc}{section}{References}
\bibliographystyle{plain}
\bibliography{BRR_Bibliography}

\begin{thebibliography}{10}

\bibitem{barbu93}
Viorel Barbu.
\newblock {\em Analysis and control of nonlinear infinite-dimensional systems},
  volume 190 of {\em Mathematics in Science and Engineering}.
\newblock Academic Press Inc., Boston, MA, 1993.

\bibitem{barbu10}
Viorel Barbu.
\newblock {\em Nonlinear differential equations of monotone types in {B}anach
  spaces}.
\newblock Springer Monographs in Mathematics. Springer, New York, 2010.

\bibitem{BDR08}
Viorel Barbu, Giuseppe Da~Prato, and Michael R{\"o}ckner.
\newblock Existence and uniqueness of nonnegative solutions to the stochastic
  porous media equation.
\newblock {\em Indiana Univ. Math. J.}, 57(1):187--211, 2008.

\bibitem{BDR09}
Viorel Barbu, Giuseppe Da~Prato, and Michael R{\"o}ckner.
\newblock Existence of strong solutions for stochastic porous media equation
  under general monotonicity conditions.
\newblock {\em Ann. Probab.}, 37(2):428--452, 2009.

\bibitem{BDR09CMP}
Viorel Barbu, Giuseppe Da~Prato, and Michael R{\"o}ckner.
\newblock Stochastic porous media equations and self-organized criticality.
\newblock {\em Comm. Math. Phys.}, 285(3), 2009.

\bibitem{BRR3}
Viorel Barbu, Michael R\"ockner, and Francesco Russo.
\newblock The stochastic porous media equation with multiplicative noise in the
  whole space.
\newblock {\em Preprint hal-00921597}.

\bibitem{BRR2}
Viorel Barbu, Michael R{\"o}ckner, and Francesco Russo.
\newblock Probabilistic representation for solutions of an irregular porous
  media type equation: the degenerate case.
\newblock {\em Probab. Theory Related Fields}, 151(1-2):1--43, 2011.

\bibitem{BR}
Nadia Belaribi, Fran\c{c}ois Cuvelier, and Francesco Russo.
\newblock A probabilistic algorithm approximating solutions of a singular pde
  of porous media type.
\newblock {\em Monte Carlo Methods Appl.}, 17(4):317--369, 2011.

\bibitem{BCR2}
Nadia Belaribi and Francesco Russo.
\newblock Uniqueness for {F}okker-{P}lanck equations with measurable
  coefficients and applications to the fast diffusion equation.
\newblock {\em Electron. J. Probab.}, 17:no. 84, 28, 2012.

\bibitem{BCRV}
S.~Benachour, P.~Chassaing, B.~Roynette, and P.~Vallois.
\newblock Processus associ\'es \`a\ l'\'equation des milieux poreux.
\newblock {\em Ann. Scuola Norm. Sup. Pisa Cl. Sci. (4)}, 23(4):793--832
  (1997), 1996.

\bibitem{BeBrC75}
P.~Benilan, H.~Brezis, and M.~G. Crandall.
\newblock A semilinear equation in {$L^1(\mathbb{R}^N)$}.
\newblock {\em Ann. Scuola Norm. Sup. Pisa Cl. Sci. (4)}, 2(4):523--555, 1975.

\bibitem{BeC81}
P.~Benilan and M.~G. Crandall.
\newblock The continuous dependence on {$\varphi $} of solutions of
  {$u\sb{t}-\Delta \varphi (u)=0$}.
\newblock {\em Indiana Univ. Math. J.}, 30(2):161--177, 1981.

\bibitem{BRR1}
P.~Blanchard, M.~R{\"o}ckner, and F.~Russo.
\newblock Probabilistic representation for solutions of an irregular porous
  media type equation.
\newblock {\em Ann. Probab.}, 38(5):1870--1900, 2010.

\bibitem{BrC79}
H.~Brezis and M.~G. Crandall.
\newblock Uniqueness of solutions of the initial-value problem for
  {$u\sb{t}-\Delta \varphi (u)=0$}.
\newblock {\em J. Math. Pures Appl. (9)}, 58(2):153--163, 1979.

\bibitem{dpz}
Giuseppe Da~Prato and Jerzy Zabczyk.
\newblock {\em Stochastic equations in infinite dimensions}, volume~44 of {\em
  Encyclopedia of Mathematics and its Applications}.
\newblock Cambridge University Press, Cambridge, 1992.

\bibitem{FRW1}
Franco Flandoli, Francesco Russo, and Jochen Wolf.
\newblock Some {SDE}s with distributional drift. {I}. {G}eneral calculus.
\newblock {\em Osaka J. Math.}, 40(2):493--542, 2003.

\bibitem{FRW2}
Franco Flandoli, Francesco Russo, and Jochen Wolf.
\newblock Some {SDE}s with distributional drift. {II}. {L}yons-{Z}heng
  structure, {I}t\^o's formula and semimartingale characterization.
\newblock {\em Random Oper. Stochastic Equations}, 12(2):145--184, 2004.

\bibitem{hushi}
Yueyun Hu and Zhan Shi.
\newblock Moderate deviations for diffusions with {B}rownian potentials.
\newblock {\em Ann. Probab.}, 32(4):3191--3220, 2004.

\bibitem{jacod}
Jean Jacod and Albert~N. Shiryaev.
\newblock {\em Limit theorems for stochastic processes}, volume 288 of {\em
  Grundlehren der Mathematischen Wissenschaften [Fundamental Principles of
  Mathematical Sciences]}.
\newblock Springer-Verlag, Berlin, second edition, 2003.

\bibitem{pages}
A.~Jakubowski, J.~M{\'e}min, and G.~Pag{\`e}s.
\newblock Convergence en loi des suites d'int\'egrales stochastiques sur
  l'espace {${\bf D}^1$} de {S}korokhod.
\newblock {\em Probab. Theory Related Fields}, 81(1):111--137, 1989.

\bibitem{KARSH}
I.~Karatzas and S.~E. Shreve.
\newblock {\em Brownian motion and stochastic calculus}, volume 113 of {\em
  Graduate Texts in Mathematics}.
\newblock Springer-Verlag, New York, second edition, 1991.

\bibitem{mathieu}
Pierre Mathieu.
\newblock On random perturbations of dynamical systems and diffusions with a
  {B}rownian potential in dimension one.
\newblock {\em Stochastic Process. Appl.}, 77(1):53--67, 1998.

\bibitem{mckean}
H.~P.~Jr. McKean.
\newblock Propagation of chaos for a class of non-linear parabolic equations.
\newblock In {\em Stochastic {D}ifferential {E}quations ({L}ecture {S}eries in
  {D}ifferential {E}quations, {S}ession 7, {C}atholic {U}niv., 1967)}, pages
  41--57. Air Force Office Sci. Res., Arlington, Va., 1967.

\bibitem{pardoux}
{\'E}tienne Pardoux.
\newblock Filtrage non lin\'eaire et \'equations aux d\'eriv\'ees partielles
  stochastiques associ\'ees.
\newblock In {\em \'{E}cole d'\'{E}t\'e de {P}robabilit\'es de {S}aint-{F}lour
  {XIX}---1989}, volume 1464 of {\em Lecture Notes in Math.}, pages 67--163.
  Springer, Berlin, 1991.

\bibitem{Ren}
Jiagang Ren, Michael R{\"o}ckner, and Feng-Yu Wang.
\newblock Stochastic generalized porous media and fast diffusion equations.
\newblock {\em J. Differential Equations}, 238(1):118--152, 2007.

\bibitem{rockpre}
Michael R{\"o}ckner and Claudia Pr{\'e}v{\^o}t.
\newblock {\em A Concise Course on Stochastic Partial Differential Equations},
  volume 1905 of {\em Lecture Notes in Mathematics}.
\newblock Springer, Berlin, 2007.

\bibitem{RTrutnau}
Francesco Russo and Gerald Trutnau.
\newblock Some parabolic {PDE}s whose drift is an irregular random noise in
  space.
\newblock {\em Ann. Probab.}, 35(6):2213--2262, 2007.

\bibitem{rv4}
Francesco Russo and Pierre Vallois.
\newblock Stochastic calculus with respect to continuous finite quadratic
  variation processes.
\newblock {\em Stochastics Stochastics Rep.}, 70(1-2):1--40, 2000.

\bibitem{russoSem}
Francesco Russo and Pierre Vallois.
\newblock Elements of stochastic calculus via regularization.
\newblock In {\em S\'eminaire de {P}robabilit\'es {XL}}, volume 1899 of {\em
  Lecture Notes in Math.}, pages 147--185. Springer, Berlin, 2007.

\bibitem{Show}
R.~E. Showalter.
\newblock {\em Monotone operators in {B}anach space and nonlinear partial
  differential equations}, volume~49 of {\em Mathematical Surveys and
  Monographs}.
\newblock American Mathematical Society, Providence, RI, 1997.

\bibitem{sv}
D.~W. Stroock and S.~R.~S. Varadhan.
\newblock {\em Multidimensional diffusion processes}.
\newblock Classics in Mathematics. Springer-Verlag, Berlin, 2006.
\newblock Reprint of the 1997 edition.

\bibitem{vazquez}
Juan~Luis V{\'a}zquez.
\newblock {\em The porous medium equation}.
\newblock Oxford Mathematical Monographs. The Clarendon Press Oxford University
  Press, Oxford, 2007.
\newblock Mathematical theory.

\end{thebibliography}

\end{document}